\documentclass[11pt]{article}
\usepackage{amsmath,amsfonts,amssymb, amsthm,enumerate,graphicx}
\usepackage{tikz,authblk}
\usepackage{subfig}
\usepackage{url}
\usepackage{hyperref}
\usepackage{stackengine,scalerel}
\usetikzlibrary {decorations.pathmorphing, decorations.pathreplacing, decorations.shapes}

\usetikzlibrary{arrows.meta}

\newtheorem{theorem} {{\textsf{Theorem}}}
\newtheorem{proposition}[theorem]{{\textsf{Proposition}}}

\newtheorem{definition}[theorem]{{\textsf{Definition}}}
\newtheorem{remark}[theorem]{{\textsf{Remark}}}

\newtheorem{lemma}[theorem]{{\textsf{Lemma}}}

\newcommand{\xdownarrow}[1]{
	{\left\downarrow\vbox to #1{}\right.\kern-\nulldelimiterspace}
}
\begin{document}

\title{A classification of semi-equivelar gems of PL $d$-manifolds on the surface with Euler characteristic $-1$}
\author{Anshu Agarwal and Biplab Basak$^1$}
	
\date{}
	
\maketitle
	
\vspace{-10mm}
\begin{center}
		
\noindent {\small Department of Mathematics, Indian Institute of Technology Delhi, New Delhi 110016, India.$^2$}

\footnotetext[1]{Corresponding author}
		
\footnotetext[2]{{\em E-mail addresses:} \url{maz228084@maths.iitd.ac.in} (A. Agarwal), \url{biplab@iitd.ac.in} (B. Basak).}
		
\medskip
		
\date{October 21, 2024}
\end{center}
	
\hrule
	
\begin{abstract}
A semi-equivelar gem of a PL $d$-manifold is a regular colored graph that represents the PL $d$-manifold and regularly embeds on a surface, with the property that the cyclic sequence of degrees of faces in the embedding around each vertex is identical.  In  \cite{bb24}, the authors classified semi-equivelar gems of PL $d$-manifolds embedded on surfaces with Euler characteristics greater than or equal to zero. In this article, we focus on classifying semi-equivelar gems of PL $d$-manifolds embedded on the surface with Euler characteristic $-1$. We prove that if a semi-equivelar gem embeds regularly on the surface with Euler characteristic $-1$, then it belongs to one of the following types: $(8^3), (6^2,8), (6^2,12), (10^2,4), (12^2,4),$ $ (4,6,14), (4,6,16), (4,6,18), (4,6,24), (4,8,10), (4,8,12),$ or $(4,8,16)$. Furthermore, we provide constructions that demonstrate the existence of such gems for each of the aforementioned types.
\end{abstract}
	


\noindent {\small {\em MSC 2020\,:} Primary  57Q15; Secondary 05C15, 05C10, 52B70, 52C20. 
		
}
	
\medskip
	
\section{Introduction}

A gem (graph encoded manifold) of a closed PL $d$-manifold is a  $(d+1)$-regular colored graph representing that manifold (cf. Subsection \ref{crystal}). It is well known that a closed PL $d$-manifold is represented by a gem, and a manifold can be represented by more than one (non-isomorphic) gem. A proof of classification of surfaces using gems can be found in \cite{b17}. A gem always admits a regular embedding on a surface \cite{ga81i}. After embedding on the surface, if the cyclic sequence of degrees of faces in the embedding around each vertex is identical, then we call that gem a semi-equivelar gem.
Basak and Binjola introduced this notion of semi-equivelar gem in \cite{bb24}, where they classified the semi-equivelar gems embedded on surfaces with non-negative Euler characteristics. This article focuses on the surface of Euler characteristic $-1$. Firstly, we compute all the possible types of semi-equivelar colored graphs embedded regularly on the surface of Euler characteristic $-1$ (see Lemma \ref{lemma:possible}). We find $15$ such possibilities: one is a $5$-regular colored graph, two are $4$-regular colored graphs, and twelve are $3$-regular colored graphs. We then observe that a semi-equivelar gem embedded regularly on the surface of Euler characteristic $-1$ cannot be a $5$-regular colored graph or a $4$-regular colored graph (see Theorem \ref{theorem:possible gem}). Finally, for all the remaining twelve embedding types, we construct $3$-regular colored gems embedded regularly on the surface of Euler characteristic $-1$ (see Theorem \ref{thm:construction}).

In \cite{BU121}, the authors classified all semi-equivelar maps on the surface with Euler characteristic $-1$ with up to 12 vertices.  Although semi-equivelar gems are inspired by semi-equivelar maps, they are distinct concepts. Specifically, semi-equivelar maps do not allow any two faces to share more than one edge, whereas semi-equivelar gems do permit such sharing. Additionally, in semi-equivelar gems, each face in the embedding is bounded by an even cycle, while there is no such restriction in semi-equivelar maps. It is important to note that our semi-equivelar gems embedded regularly on the surface with Euler characteristic $-1$  are entirely different from the semi-equivelar maps  on the surface with Euler characteristic $-1$  described in  \cite{BU121}. For further reading on semi-equivelar maps and semi-equivelar gems, refer to \cite{bb24, bk08, dm18,dm22}. 
	
\section{Preliminaries}
The theory of edge-colored graphs provides a method to represent any piecewise-linear (PL) manifold. It is known that any closed connected PL $d$-manifold can be described by a $(d+1)$-regular colored graph, where loops are not allowed. This approach offers a framework for converting between manifolds and graphs, allowing us to study manifolds through their corresponding graphical representations.
	
\subsection{Graph encoded manifolds (gem)} \label{crystal}
	
For a multigraph $\Gamma= (V(\Gamma), E(\Gamma))$ without loops, the edges are labeled (or colored) by $\Delta_d =\{0,1, \dots, d\}$. The coloring is called a {\it proper edge-coloring} if any two adjacent edges have different colors. The members of the set $\Delta_d$ are called the {\it colors} of $\Gamma$. More precisely, for a proper edge-coloring, there exists a surjective map $\gamma: E(\Gamma) \to \Delta_d$ with  $\gamma(e_1) \ne \gamma(e_2)$ for any two adjacent edges $e_1$ and $e_2$. A graph with a proper edge coloring is denoted by $(\Gamma,\gamma)$.
If the degree of each vertex in a graph $\Gamma$ is  $(d+1)$, then it is said to be {\it $(d+1)$-regular}. We refer to \cite{bm08} for standard terminology on graphs. All spaces will be considered in the PL-category.

A  {\it $(d+1)$-regular colored graph} is a pair $(\Gamma,\gamma)$, where $\Gamma$ is a $(d+1)$-regular graph and $\gamma$ is a proper edge-coloring. If there is no confusion with coloration,  $\Gamma$ can be used instead of $(\Gamma,\gamma)$ for $(d+1)$-regular colored graphs. 
%
For each  $(d+1)$-regular colored graph  $(\Gamma,\gamma)$, a corresponding $d$-dimensional simplicial cell-complex ${\mathcal K}(\Gamma)$ is constructed as follows:
	
\begin{itemize}
\item{} for each vertex $u\in V(\Gamma)$, take a $d$-simplex $\sigma(u)$ with vertices labeled by $\Delta_d$;
		
\item{} corresponding to each edge of color $j$ between $u,v\in V(\Gamma)$, identify the ($d-1$)-faces of $\sigma(u)$ and $\sigma(v)$ opposite to $j$-labeled vertices such that the same labeled vertices coincide.
\end{itemize}
If the geometric carrier $|\mathcal{K}(\Gamma)|$  is (PL) homeomorphic to a PL $d$-manifold $M$ then $\mathcal{K}(\Gamma)$ is said to be a {\it colored triangulation} of $M$, and $(\Gamma,\gamma)$ is said to be a {\it gem} (graph encoded manifold) of $M$ (or is said to represent $M$). Clearly, every 3-regular colored graph represents a closed connected surface. It is well known that every closed connected PL $d$-manifold admits a gem. Let $\Gamma$ be a gem representing a manifold $M$. Then, $\Gamma$ is bipartite if and only if $M$ is orientable.
	
Let $(\Gamma,\gamma)$ and $(\bar{\Gamma},\bar{\gamma})$ be two  $(d+1)$-regular colored graphs with color sets $\Delta_d$ and $\bar{\Delta}_d$, respectively. Then $I:=(I_V,I_c):\Gamma \to \bar{\Gamma}$ is called an {\em isomorphism} if $I_V: V(\Gamma) \to V(\bar{\Gamma})$ and $I_c:\Delta_d \to \bar{\Delta}_d$ are bijective maps such that $uv$ is an edge of color $i \in \Delta_d$ if and only if $I_V(u)I_V(v)$ is an edge of color $I_c(i) \in \bar{\Delta}_d$. The graphs $(\Gamma, \gamma)$ and $(\bar{\Gamma}, \bar{\gamma})$ are then said to be {\it isomorphic}.

\subsection{Regular embedding} 

Let $\Gamma$ be a $(d+1)$-regular colored graph. We say that $\Gamma$ embeds regularly on a surface $S$ if it can be embedded on $S$ in such a way that each face of the embedding is bounded by a bi-colored cycle, where the cycle uses two consecutive colors, $\varepsilon_i$ and $\varepsilon_{i+1}$ for some $i$, and indices are taken modulo $d+1$. Here, $\varepsilon = (\varepsilon_0, \dots, \varepsilon_d)$ represents a cyclic permutation of $\Delta_d$. Regular embeddings are of significant interest in combinatorial topology, and numerous notable results on this topic can be found in \cite{b19, bb21, bc17, cp90, fg82, g81, ga81i}. Below, we present several key results from \cite{cp90, ga81i} that are relevant to regular embeddings and will be useful for our article.

\begin{proposition}[\cite{ga81i}]
If $\Gamma$ is a bipartite (respectively, non-bipartite) $(d+1)$-regular colored graph representing a closed connected orientable (respectively, non-orientable) PL $d$-manifold $M$, then for each cyclic permutation $\varepsilon = (\varepsilon_0, \dots, \varepsilon_d)$ of $\Delta_d$, there exists a regular embedding of $\Gamma$ into an orientable (respectively, non-orientable) surface $S$.  
\end{proposition}

\begin{proposition}[\cite{ga81i}]\label{prop: non-exist embedding}
 A bipartite graph cannot be embedded regularly on a non-orientable surface, and a non-bipartite graph cannot be embedded regularly on an orientable surface.   
\end{proposition}

\begin{proposition}[\cite{ga81i}]\label{prop: surface embedding}
A 3-regular colored graph $\Gamma$ represents a closed connected surface $S$ if and only if it embeds regularly on the surface $S$.   
\end{proposition}

\begin{proposition}[\cite{cp90}]\label{prop: even}
If a $(d+1)$-regular colored graph representing a closed connected non-orientable PL $d$-manifold (for $d \geq 3$) embeds regularly on a surface $S$, then $S$ must be a non-orientable surface with an even genus.   
\end{proposition}

Let $\Gamma$ be a $(d+1)$-regular colored graph embedded regularly on a surface $S$. Since each face in the embedding is bounded by a bi-colored cycle, every face forms a polygon with an even number of sides (including the possibility of a 2-gon). In this article, we restrict our focus to polygons of length at least 4. Without loss of generality, we assume the color sequence $\varepsilon$ to be $(0, 1, \dots, d)$. We define the face-cycles $P_0, P_1, \dots, P_d$ at a vertex $x$ in the embedding of $\Gamma$ on $S$ as the consecutive faces incident to $x$, where each polygon $P_i$ is bounded by a bi-colored cycle of colors $i$ and $i+1$, for $0 \leq i \leq d$, with the condition that $d+1 = 0$.
	   
\begin{definition}\label{def:semiequivelar}
{\rm
Let $\Gamma$ be a  $(d+1)$-regular colored graph embedded regularly on a surface $S$. If the face-cycles $P_0, P_1, \dots, P_d$ at every vertex are of the same type  in the embedding of $\Gamma$ on the surface $S$, then $\Gamma$ is called a {\it semi-equivelar graph embedded regularly on  $S$}. Furthermore, if $\Gamma$ represents a $d$-manifold $M$, then it is referred to as a {\it semi-equivelar gem of $M$.}}

{\rm If there are $n_i$ adjacent $p_i$-gons and $p$ is the total number of vertices in $\Gamma$, then $\Gamma$ is called a {\it $[(p_0^{n_0},p_1^{n_1},\dots,p_m^{n_m});p]$-type semi-equivelar graph embedded regularly on  $S$}. It is important to note that $p_0\neq p_m$ and $p_i$ may be equal to $p_j$  if $|i-j|\geq 2$.}

\end{definition}

\section{Main results}
In \cite{bb24}, the authors examined all semi-equivelar gems that are  embedded regularly on surfaces with non-negative Euler characteristics. This article will focus on the surface with Euler characteristic $-1$, i.e., the surface $\#_3 \mathbb{RP}^2$. The following lemma lists all the possible types of semi-equivelar graphs embedded regularly  on the surface with Euler characteristic $-1$, i.e., the surface $\#_3 \mathbb{RP}^2$.

\begin{lemma}\label{lemma:possible}
If $\Gamma$ is a semi-equivelar graph embedded regularly on the surface $S$ with $\chi(S)=-1$, then $\Gamma$ is one of the following types: $[(4^5);4],\ [(4^3,8);8],\ [(4^3,6);12],\ [(8^3);8],\ [(6^2,8);24],\ $ $[(6^2,12);12],\ [(10^2,4);20],\ [(12^2,4);12],\ [(4,6,14);84],\ [(4,6,16);48],\ [(4,6,18);36],\ $\\
$ [(4,6,24);24],\ [(4,8,10);40],\ [(4,8,12);24]$, or $[(4,8,16);16].$ 

\end{lemma}

\begin{proof}
Let $\Gamma$ be a $[({p_0},{p_1},\dots,{p_d});p]$-type semi-equivelar graph embedded regularly on $S$, where $p_i \geq 4$. Clearly, $p_i$ is even. Let $q_0,q_1, \dots, q_l$ be the lengths of the polygons of different sizes, where $q_j=p_i$, for some $0\leq i \leq d$. Let $k_i$ be the number of $q_i$-polygons. Thus, $\sum_{i=0}^l{k_i}=d+1$. Let $V, E$, and $F$ denote the number of vertices, edges, and faces in the regular embedding of $\Gamma$ on $S$, respectively. Then $V=p$, $E=p (d+1)/2$, and $F=p(\frac{k_0}{q_0}+\frac{k_1}{q_1}+ \dots +\frac{k_l}{q_l})$. Thus, we have
\begin{align}
\Big(1-\frac{(d+1)}{2}+\frac{k_0}{q_0}+\frac{k_1}{q_1}+ \dots +\frac{k_l}{q_l}\Big)=\frac{\chi(S)}{p}. \label{1}
\end{align}
Since $q_i \geq 4$, we have $k_i/q_i \leq k_i/4$, which further implies
\begin{equation} \label{2}
d+1 \leq 4-\frac{4 \chi(S)}{p}=4+\frac{4}{p}. 
\end{equation}

\noindent Since $p\geq 4, \ (d+1)\leq 5.$ 
		
\noindent \textbf {Case 1.} Let $(d+1)=5$. Equation \eqref{2} gives $p=4$, and hence $p_i=4$ for all $0 \leq i \leq 4$. Thus, a $5$-regular colored semi-equivelar graph embedded regularly on the surface $\#_3 \mathbb{RP}^2$ is of the type $[(4^5);4]$.

\noindent \textbf {Case 2.} Let $(d+1)=4$. In this case, $(k_0,k_1,\dots,k_l)=(4),(3,1),(2,2),(2,1,1)$, or $(1,1,1,1)$.
Also, Equation \eqref{1} reduces to 
\begin{align}\label{3}
\frac{k_0}{q_0}+\frac{k_1}{q_1}+ \dots +\frac{k_l}{q_l}=1-\frac{1}{p}.
\end{align}

Consider $(k_0)=(4)$. From Equation \eqref{3}, it is clear that $q_0 \ne 4$, and when $q_0\geq 6$, we get $p\leq 3$, which is not possible. So, for $(k_0)=(4)$, there is no semi-equivelar graph embedded regularly on the surface with Euler characteristic $-1$.

For $(k_0,k_1)=(3,1)$, Equation \eqref{3} implies $\frac{3}{q_0}+\frac{1}{q_1}+\frac{1}{p}=1.$ First, let $q_0 < q_1$. If $q_0 \geq 6$, then $1=\frac{3}{q_0}+\frac{1}{q_1}+\frac{1}{p}\leq \frac{3}{6}+\frac{1}{8}+\frac{1}{8}=\frac{3}{4}< 1$. Therefore, $q_0\geq 6$ is not possible. Now, if $q_0=4$, we get $\frac{1}{q_1}+\frac{1}{p}=\frac{1}{4}$. Putting $q_1=6,8$ in this equation, we get $p=12,8$, respectively. Thus, we got two embedding types $[(4^3,6);12]$ and $[(4^3,8);8]$. Again, using the same drill for $q_1\geq 10$, one gets $\frac{1}{4}\leq \frac{1}{5}$, which implies $q_1\geq 10$ is not possible. Now, let $q_1 < q_0$. If $q_1 \geq 6$, then $1=\frac{3}{q_0}+\frac{1}{q_1}+\frac{1}{p}\leq \frac{3}{8}+\frac{1}{6}+\frac{1}{8}=\frac{2}{3}< 1$. Therefore, $q_1\geq 6$ is not possible. So, we are only left with the possibility $q_1=4$. Taking $q_0=6$, we get $p=4<q_0$, which is a contradiction. And, for $q_0\geq 8$, we get $1\leq \frac{3}{8}+\frac{1}{4}+\frac{1}{8}= \frac{3}{4}$, which implies that for $q_1 < q_0$, we do not have any semi-equivelar graph embedded regularly on the surface with Euler characteristic $-1$.

For $(k_0,k_1)=(2,2)$, Equation \eqref{3} reduces to $\frac{2}{q_0}+\frac{2}{q_1}+\frac{1}{p}=1.$ We can assume $q_0<q_1$. First, let $q_0=4$ and $q_1=6$, then we get $p=6$ from the preceding equation. But this is not possible, since $q_0=4$ does not divide $p=6$. Now, letting $q_0\geq 6$, we get $1=\frac{2}{q_0}+\frac{2}{q_1}+\frac{1}{p}\leq \frac{2}{6}+\frac{2}{8}+\frac{1}{8}=\frac{17}{24}<1$. So, $q_0\geq 6$ is not possible. Hence, for $(k_0,k_1)=(2,2)$, we do not get any semi-equivelar graph embedded regularly on the surface $\#_3 \mathbb{RP}^2.$

For $(k_0,k_1,k_2)=(2,1,1)$, Equation \eqref{3} reduces to $\frac{2}{q_0}+\frac{1}{q_1}+\frac{1}{q_2}+\frac{1}{p}=1.$ Let $q_0$ be the smallest, then preceding equation implies $1=\frac{2}{q_0}+\frac{1}{q_1}+\frac{1}{q_2}+\frac{1}{p}\leq \frac{2}{4}+\frac{1}{6}+\frac{1}{8}+\frac{1}{8}=\frac{11}{12}<1$, leading us to a contradiction. So, for this $(k_0,k_1,k_2)=(2,1,1)$ also, we do not get any semi-equivelar graph embedded regularly on $S$.

For $(k_0,k_1,k_2,k_3)=(1,1,1,1)$, Equation \eqref{3} reduces to $\frac{1}{q_0}+\frac{1}{q_1}+\frac{1}{q_2}+\frac{1}{q_3}+\frac{1}{p}=1.$ We can assume $q_0$ to be the smallest. Then preceding equation implies $1=\frac{1}{q_0}+\frac{1}{q_1}+\frac{1}{q_2}+\frac{1}{q_3}+\frac{1}{p}\leq \frac{1}{4}+\frac{1}{6}+\frac{1}{8}+\frac{1}{10}+\frac{1}{10}=\frac{89}{120}<1$, which is a contradiction. So, for this $(k_0,k_1,k_2,k_3)=(1,1,1,1)$ as well, we do not have any semi-equivelar graph embedded regularly on the surface $\#_3 \mathbb{RP}^2.$

Thus, a $4$-regular colored semi-equivelar graph embedded regularly on the surface $\#_3 \mathbb{RP}^2$ is of the type $[(4^3,6);12]$ or $[(4^3,8);8].$

\noindent \textbf {Case 3.} Let $(d+1)=3$. So, in this case, $(k_0,k_1,\dots,k_l)=(3),(2,1)$, or $(1,1,1)$, and Equation \eqref{1} implies that
\begin{equation}\label{4}
\frac{k_0}{q_0}+\frac{k_1}{q_1}+ \dots +\frac{k_l}{q_l}=\frac{1}{2}-\frac{1}{p}.
\end{equation}

For $(k_0)=3,$ Equation \eqref{4} gives $\frac{3}{q_0}+\frac{1}{p}=\frac{1}{2}.$ Clearly, $q_0=4,6$ is not possible. Putting $q_0=8$ in the equation, we get $p=8$. So, $[(8^3);8]$ is a possible embedding type. Now, if $q_0\geq 10$, then $\frac{1}{2}=\frac{3}{q_0}+\frac{1}{p} \leq \frac{3}{10}+\frac{1}{10}=\frac{4}{10}< \frac{1}{2}$, which is absurd. So, for $(k_0)=3,$ the only possible embedding type is $[(8^3);8]$.  

For $(k_0,k_1)=(2,1)$, Equation \eqref{4} implies
\begin{equation}\label{5}
\frac{2}{q_0}+\frac{1}{q_1}+\frac{1}{p}=\frac{1}{2}.    
\end{equation}
Clearly, $q_0=4$ is not possible. So, letting $q_0=6$, Equation \eqref{5} implies $\frac{1}{q_1}+\frac{1}{p}=\frac{1}{6}$. From this equation, it is clear that $q_1$ is not equal to $4$ and is not greater than or equal to $14$, else $\frac{1}{6}\leq\frac{1}{7}$, which is absurd. Putting $q_1=8,10,12$ in the above equation, we get $p=24,15,12$, respectively. Since we are considering only regular colored graphs, $p$ is even. Thus, $p=15$ is discarded. Hence, when $q_0=6$, we get two possible embedding types of semi-equivelar graphs $[(6^2,8);24]$ and $[(6^2,12);12].$ Now, if $q_0=8$ is fixed, then Equation \eqref{5} reduces to $\frac{1}{q_1}+\frac{1}{p}=\frac{1}{4}$. Clearly, this equation implies that $q_1\neq 4$. Putting $q_1=6$ in the above equation, we get $p=12$. But this is not possible, since $q_0=8$ does not divide $12$. Also, $q_1\geq 10$ is not possible, otherwise $\frac{1}{4}=\frac{1}{q_1}+\frac{1}{p}\leq \frac{1}{5}$, which is obviously not possible. Thus, for $q_0=8$, we do not have any possible embedding type. Again, let us fix $q_0=10$. Then, Equation \eqref{5} implies $\frac{1}{q_1}+\frac{1}{p}=\frac{3}{10}.$ Clearly, from this equation, we get that $q_1 \geq 6$ is not possible. Because $\frac{3}{10}=\frac{1}{q_1}+\frac{1}{p}\leq \frac{1}{6}+\frac{1}{10}=\frac{8}{30}$, leads us to a contradiction. Putting $q_1=4$, we get $p=20$. Thus, when we fix $q_0=10$, we get a possible embedding type $[(10^2,4);20]$. In a similar way, one can show that if we fix $q_0=12$, then $[(12^2,4);12]$ is the unique possible embedding type.
Hence, for $(k_0,k_1)=(2,1)$, we get four possible embedding types of semi-equivelar graphs. These are $[(6^2,8);24],\ [(6^2,12);12],\ [(10^2,4);20]$, and $[(12^2,4);12]$.

For $(k_0,k_1,k_2)=(1,1,1)$, Equation \eqref{4} implies
\begin{equation}\label{6}
\frac{1}{q_0}+\frac{1}{q_1}+\frac{1}{q_2}+\frac{1}{p}=\frac{1}{2}.
\end{equation}
Without loss of generality, we assume $q_0 < q_1 < q_2$. Equation \eqref{6} implies that $q_0\geq6$ is not possible, as if it is, then $\frac{1}{2}=\frac{1}{q_0}+\frac{1}{q_1}+\frac{1}{q_2}+\frac{1}{p}\leq \frac{1}{6}+\frac{1}{8}+\frac{1}{10}+\frac{1}{10}=\frac{59}{120}$, which is absurd. So, we get $q_0=4$, and Equation \eqref{6} further implies 
\begin{equation}\label{7}
\frac{1}{q_1}+\frac{1}{q_2}+\frac{1}{p}=\frac{1}{4}.
\end{equation}
The above equation implies that $q_1\geq 10$ and $q_2\geq 14$ is not possible, else $\frac{1}{4}=\frac{1}{q_1}+\frac{1}{q_2}+\frac{1}{p}\leq \frac{1}{10}+\frac{1}{14}+\frac{1}{14} \leq \frac{17}{70}$. Putting $q_1=10$ and $q_2=12$ in Equation \eqref{7}, we get $p=15$. Since we consider only regular colored graphs, $p$ must be even. So, $p=15$ is not possible, and thus, we get $q_1\leq 8.$ Now, let us first assume $q_1=6$, then Equation \eqref{7} implies $\frac{1}{q_2}+\frac{1}{p}=\frac{1}{12}$. Clearly, from this equation, $14 \leq q_2 \leq 24$. Putting $q_2=14,16,18,20,22,24$ in the above equation, we get $p=84,48,36,30,\frac{132}{5},24$, respectively. We discard $p=30$ because $20$ does not divide $30$, and $p=\frac{132}{5}$ is obviously not possible. So, fixing $q_0=4$ and $q_0=6$, we get four possible types $[(4,6,14);84],\ [(4,6,16);48],\ [(4,6,18);36]$, and $[(4,6,24);24]$. Now, let $q_1=8$, then Equation \eqref{7} implies $\frac{1}{q_2}+\frac{1}{p}=\frac{1}{8}.$ This implies $10 \leq q_2 \leq 16$. Putting $q_2=10,12,14,16$ in the above equation, we get $p=40,24,\frac{56}{3},16$, respectively. Discarding $p=\frac{56}{3}$, we get three possible embedding types $[(4,8,10);40],\ [(4,8,12);24]$, and $[(4,8,16);16]$.

Thus, a $3$-regular colored semi-equivelar graph
embedded regularly on the surface $\#_3 \mathbb{RP}^2$ is one of the following twelve types: $[(8^3);8],\ [(6^2,8);24],\ [(6^2,12);12],\ [(10^2,4);20],\ [(12^2,4);12],\ $ $[(4,6,14);84],\ [(4,6,16);48],\ [(4,6,18);36],\ [(4,6,24);24],\ [(4,8,10);40],\ [(4,8,12);24]$, and $[(4,8,16);16]$.		
\end{proof}

\begin{theorem}\label{theorem:possible gem}
Let $\Gamma$ be a semi-equivelar gem embedded regularly on the surface with Euler characteristic $-1$. Then, $\Gamma$ represents $\#_3 \mathbb{RP}^2$, and  $\Gamma$ is one of the following twelve types: $[(8^3);8],\ [(6^2,8);24],\ [(6^2,12);12],\ [(10^2,4);20],\ [(12^2,4);12],\ $ $[(4,6,14);84],\ [(4,6,16);48]$, $[(4,6,18);36],\ [(4,6,24);24],\ [(4,8,10);40],\ [(4,8,12);24]$, and $[(4,8,16);16]$.		
\end{theorem}

\begin{proof}
Let $\Gamma$ be a gem representing a closed connected PL $d$-manifold $M$, which is  embedded regularly on the surface with Euler characteristic $-1$. By Proposition \ref{prop: non-exist embedding}, it follows that $M$ must be a non-orientable manifold. Furthermore, Proposition \ref{prop: even} implies that $d=2$. Consequently, by Proposition \ref{prop: surface embedding}, $M$ is the surface $\#_3 \mathbb{RP}^2$. Now, the result follows from Lemma \ref{lemma:possible}.    
\end{proof}

\begin{remark}\label{remark:non-existence}
{\rm There are no semi-equivelar gems of types $(4^5)$, $(4^3,8)$, and $(4^3,6)$ that embed regularly on the surface  with Euler characteristic $-1$.  For an example, Figure \ref{fig:13} depicts the unique semi-equivelar graph of type $(4^5)$ with four vertices that embeds regularly on the surface with Euler characteristic $-1$. However, it is not a gem, as it represents a topological space that is not a manifold.}
\end{remark}

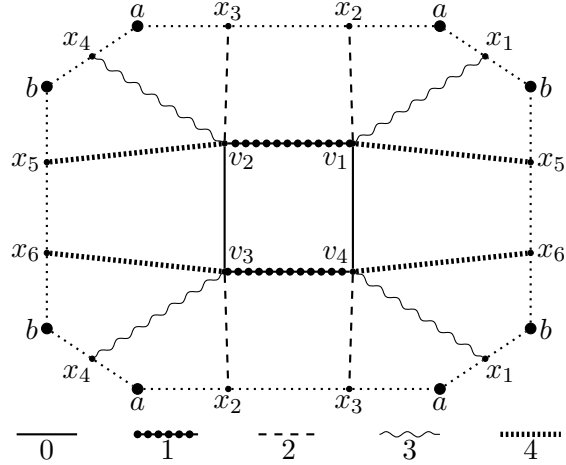
\begin{figure}[h!]
\tikzstyle{ver}=[]
\tikzstyle{vert}=[circle, draw, fill=black!100, inner sep=0pt, minimum width=2pt]
\tikzstyle{verti}=[circle, draw, fill=black!100, inner sep=0pt, minimum width=4pt]
\tikzstyle{edge} = [draw,thick,-]
    \centering
\begin{tikzpicture}[scale=0.4]

\begin{scope}
\foreach \x/\y/\z in
{6.5/5/1,2/6/2,-2/6/3,-6.5/5/4,-8/1.5/5,-8/-1.5/6,-6.5/-5/7,-2/-6/8,2/-6/9,6.5/-5/10,8/-1.5/11,8/1.5/12}
{\node[vert] (\z) at (\x,\y){};}

\foreach \x/\y/\z in
{7/5.5/x_1,2/6.5/x_2,-2/6.5/x_3,-7/5.5/x_4,-8.7/1.5/x_5,-8.7/-1.5/x_6,-7/-5.5/x_4,-2/-6.5/x_2,2/-6.5/x_3,7/-5.5/x_1,8.7/-1.5/x_6,8.7/1.5/x_5}
{\node[ver] (\z) at (\x,\y){$\z$};}

\foreach \x/\y/\z in
{5/6/1',-5/6/2',-8/4/3',-8/-4/4',-5/-6/5',5/-6/6',8/-4/7',8/4/8'}
{\node[verti] (\z) at (\x,\y){};}

\foreach \x/\y/\z in
{5/6.5/a,-5/6.5/a,-8.5/4/b,-8.5/-4/b,-5/-6.5/a,5/-6.5/a,8.5/-4/b,8.5/4/b}
{\node[ver] (\z) at (\x,\y){$\z$};}

\end{scope}

\begin{scope}[shift={(0,0)}]
\foreach \x/\y in {45/x1,135/x2,225/x3,315/x4}
{\node[vert] (\y) at (\x:3){};}

\foreach \x/\y in {45/v_1,135/v_2,225/v_3,315/v_4}
{\node[ver] at (\x:2.2){$\y$};}
\end{scope}

\foreach \x/\y in {1/1',1'/2,2/3,3/2',2'/4,4/3',3'/5,5/6,6/4',4'/7,7/5',5'/8,8/9,9/6',6'/10,10/7',7'/11,11/12,12/8',8'/1}
{\path[edge,dotted] (\x) -- (\y);}

\foreach \x/\y in {x1/x2,x2/x3,x3/x4,x4/x1}
{\path[edge] (\x) -- (\y);}

\foreach \x/\y in {x1/x2,x3/x4}
{{\draw [line width=3pt, line cap=round, dash pattern=on 0pt off 1.3\pgflinewidth]  (\x) -- (\y);}}

\foreach \x/\y in {x1/2,x2/3,x3/8,x4/9}
{\path[edge,dashed] (\x) -- (\y);}

\foreach \x/\y in {x1/1,x2/4,x3/7,x4/10}
{\draw[decorate,decoration={snake, amplitude=1pt, segment length=8pt}] (\x) -- (\y);}

\foreach \x/\y in {x1/12,x2/5,x3/6,x4/11}
{\draw[line width=2pt, line cap=rectangle, dash pattern=on 1pt off 1] (\x) -- (\y);}

\begin{scope} [scale=1, shift = {(-9, -7.5)}]
\foreach \x/\y/\z in {1/-0.5/0,5/-0.5/1,9/-0.5/2,13/-0.5/3,17/-0.5/4}
{\node[ver] () at (\x,\y){$\z$};}
\path[edge] (0,0) -- (2,0);
\path[edge] (4,0) -- (6,0);
\draw [line width=3pt, line cap=round, dash pattern=on 0pt off 1.3\pgflinewidth]  (4,0) -- (6,0);
\path[edge, dashed] (8,0) -- (10,0);
\draw[decorate,decoration={snake, amplitude=1pt, segment length=8pt}] (12,0) -- (14,0);
\draw[line width=2pt, line cap=rectangle, dash pattern=on 1pt off 1] (16,0) -- (18,0);

\end{scope}  

 \end{tikzpicture}

 \caption{Only possible semi-equivelar graph of type $(4^5)$ embedded regularly on $\#_{3} \mathbb{RP}^2$.}  \label{fig:13}
\end{figure}

\begin{theorem}\label{thm:construction}
For each of the following types: $(8^3), (6^2,8), (6^2,12), (10^2,4), (12^2,4),$ $ (4,6,14),$ $ (4,6,16), (4,6,18), (4,6,24), (4,8,10), (4,8,12),$ or $(4,8,16)$, there exists a semi-equivelar gem that embeds regularly on the surface with Euler characteristic $-1$. Further, each of the gems represents the surface $\#_3 \mathbb{RP}^2$.
\end{theorem}

\begin{proof}
 In Figures \ref{fig:1} through \ref{fig:12}, we present a CW-complex structure of a surface with precisely one 2-cell. The boundary of the 2-cell contains at most seven 0-cells from the set $\{a, b, c, d, e, f, g\}$. The 1-cells are depicted by the dotted lines on the boundary of the 2-cell, with the identification of two 1-cells in a manner that the $x_i$'s are being identified. We will now provide a detailed description of each figure. Recall that every 3-regular colored graph represents a closed connected surface.

\begin{figure}[h!]
\tikzstyle{ver}=[]
\tikzstyle{vert}=[circle, draw, fill=black!100, inner sep=0pt, minimum width=2pt]
\tikzstyle{verti}=[circle, draw, fill=black!100, inner sep=0pt, minimum width=4pt]
\tikzstyle{edge} = [draw,thick,-]
    \centering
\begin{tikzpicture}[scale=0.5]
\begin{scope}

\foreach \x/\y in {45/1,90/2,135/3,180/4,225/5,270/6,315/7,0/8}
{\node[vert] (\y) at (\x:3){};}

\foreach \x/\y in {45/v_1,90/v_2,135/v_3,180/v_4,225/v_5,270/v_6,315/v_7,0/v_8}
{\node[ver] at (\x:2.5){$\y$};}
							
\foreach \x/\y in {22.5/1',67.5/2',112.5/3',157.5/4',202.5/5',247.5/6',292.5/7',337.5/8'}
{\node[verti] (\y) at (\x:6){};}

\foreach \x/\y in {45/1'',90/2'',135/3'',180/4'',225/5'',270/6'',315/7'',0/8''}
{\node[vert] (\y) at (\x:6){};}

\foreach \x/\y in {22.5/a,67.5/b,112.5/a,157.5/b,202.5/a,247.5/b,292.5/a,337.5/b}
{\node[ver] at (\x:6.5){$\y$};}

\foreach \x/\y in {45/x_3,90/x_1,135/x_2,180/x_4,225/x_2,270/x_1,315/x_3,0/x_4}
{\node[ver] (\y) at (\x:6.5){$\y$};}

\foreach \x/\y in {1'/1'',1''/2',2'/2'',2''/3',3'/3'',3''/4',4'/4'',4''/5',5'/5'',5''/6',6'/6'',6''/7',7'/7'',7''/8',8'/8'',8''/1'}
{\path[edge,dotted] (\x) -- (\y);}

\foreach \x/\y in {1/1'',2/2'',3/3'',4/4'',5/5'',6/6'',7/7'',8/8''}
{\path[edge,dashed] (\x) -- (\y);}

\foreach \x/\y in {1/2,2/3,3/4,4/5,5/6,6/7,7/8,8/1}
{\path[edge] (\x) -- (\y);}

\foreach \x/\y in {2/3,4/5,6/7,8/1}
{\draw [line width=3pt, line cap=round, dash pattern=on 0pt off 1.3\pgflinewidth]  (\x) -- (\y);}

\end{scope}   

\begin{scope} [scale=1, shift = {(-5, -7.3)}]
\foreach \x/\y/\z in {1/-0.5/0,5/-0.5/1,9/-0.5/2}
{\node[ver] () at (\x,\y){$\z$};}
\path[edge] (0,0) -- (2,0);
\path[edge] (4,0) -- (6,0);
\draw [line width=3pt, line cap=round, dash pattern=on 0pt off 1.3\pgflinewidth]  (4,0) -- (6,0);
\path[edge, dashed] (8,0) -- (10,0);

\end{scope}  
    \end{tikzpicture}

 \caption{Embedding on $\#_3 \mathbb{RP}^2$ of gem representing $\#_{3} \mathbb{RP}^2$ of type $(8^3)$.}
    \label{fig:1}
\end{figure}
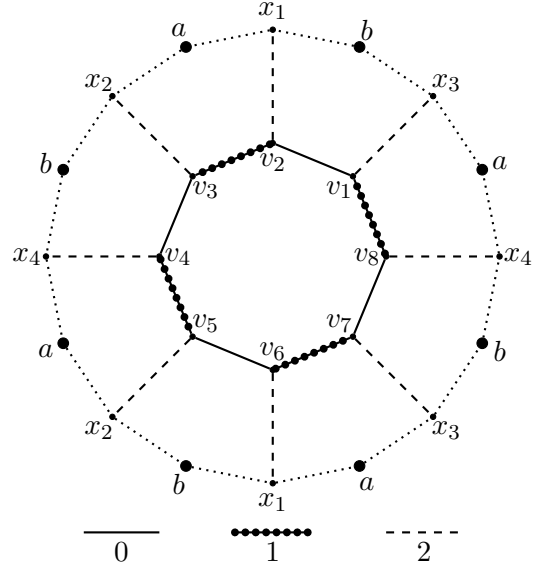


\noindent \textbf{Figure \ref{fig:1}.} In this diagram, the surface is composed of two 0-cells labeled $a$ and $b$, four 1-cells, and one 2-cell. Consequently, the Euler characteristic of the surface is $-1$. The $\{0,1\}$-colored cycle bounds the inner octagonal face $v_1v_2v_3v_4v_5v_6v_7v_8$. The $\{1,2\}$-colored $8$-cycle bounds the face $v_1v_8x_4v_4v_5x_2v_3v_2x_1v_6v_7x_3$, and the 0-cell $a$ lies in the interior of this face. The $\{0,2\}$-colored $8$-cycle bounds the face $v_1v_2x_1v_6v_5x_2v_3v_4x_4v_8v_7x_3$, and the 0-cell $b$ lies in the interior of this face. Thus, the 3-regular colored graph in Figure \ref{fig:1} is a semi-equivelar gem of type $[(8^3);8]$ embedded regularly on the surface $\#_3 \mathbb{RP}^2$.

\noindent \textbf{Figure \ref{fig:2}.} In this figure, the surface exhibits seven $0$-cells labeled $a$, $b$, $c$, $d$, $e$, $f$, and $g$, along with nine $1$-cells, and one $2$-cell. Hence, the Euler characteristic of the surface is $-1$. The $\{0,2\}$-colored three $8$-cycles $A_1,A_2,A_3$ bound the inner octagonal faces $v_1v_2v_3v_4v_5v_6v_7v_8, v_9v_{10}v_{11}v_{12}$ $v_{13}v_{14}v_{15}v_{16}$, and $v_{17}v_{18}v_{19}v_{20}v_{21}v_{22}v_{23}v_{24}$, respectively. The $\{1,2\}$-colored four $6$-cycles $B_1,$ $B_2,B_3$, and $B_4$ bound the faces $v_1x_6v_{24}v_{23}x_9v_{13}v_{12}x_3v_2$ containing $c$, $v_3x_2v_{5}v_{6}v_{10}v_{11}x_1v_4$ containing $a$, $v_{14}x_7v_{16}v_{9}v_{21}v_{22}x_8v_{15}$ containing $f$, and $v_{17}x_5v_{19}v_{20}v_{7}v_{8}x_4v_{18}$ containing $d$, respectively. The $\{0,1\}$-colored four $6$-cycles $C_1,C_2,C_3,C_4$ bound the faces $v_1x_6v_{24}v_{17}x_5v_{19}v_{18}x_4v_8$ containing $e$, $v_3x_2v_{5}v_{4}x_1v_{11}v_{12}x_3v_2$ containing $b$, $v_{14}x_7v_{16}v_{15}x_8v_{22}v_{23}x_9v_{13}$ containing $g$, and the inner face $v_{6}v_{7}v_{20}v_{21}v_{9}v_{10}$, respectively. Thus, the 3-regular colored graph in Figure \ref{fig:2} is a semi-equivelar gem of type $[(6^2,8);24]$ embedded regularly on the surface $\#_3 \mathbb{RP}^2$.

\begin{figure}[h!]
\tikzstyle{ver}=[]
\tikzstyle{vert}=[circle, draw, fill=black!100, inner sep=0pt, minimum width=2pt]
\tikzstyle{verti}=[circle, draw, fill=black!100, inner sep=0pt, minimum width=4pt]
\tikzstyle{edge} = [draw,thick,-]
    \centering
\begin{tikzpicture}[scale=0.6]

\begin{scope}
\foreach \x/\y in {0/1,20/2,40/3,60/4,80/5,100/6,120/7,140/8,160/9,180/10,200/11,220/12,240/13,260/14,280/15,300/16,320/17,340/18}
{\node[vert] (\y) at (\x:7){};}

\foreach \x/\y in {10/1',30/2',50/3',70/4',90/5',110/6',130/7',150/8',170/9',190/10',210/11',230/12',250/13',270/14',290/15',310/16',330/17',350/18'}
{\node[verti] (\y) at (\x:7){};}

\foreach \x/\y/\z in 
{0/1/x_5,20/2/x_4,40/3/x_6,60/4/x_3,80/5/x_2,100/6/x_1,120/7/x_2,140/8/x_1,160/9/x_3,180/10/x_9,200/11/x_7,220/12/x_8,240/13/x_7,260/14/x_8,280/15/x_9,300/16/x_6,320/17/x_5,340/18/x_4}
{\node[ver] () at (\x:7.5){$\z$};}

\foreach \x/\y/\z in {10/1'/d,30/2'/e,50/3'/c,70/4'/b,90/5'/a,110/6'/b,130/7'/a,150/8'/b,170/9'/c,190/10'/g,210/11'/f,230/12'/g,250/13'/f,270/14'/g,290/15'/c,310/16'/e,330/17'/d,350/18'/e} {\node[ver] () at (\x:7.5){$\z$};}
\end{scope} 

\begin{scope}[shift = {(0, 3)}, rotate=20 ]
\foreach \x/\y in {0/x1,45/x2,90/x3,135/x4,180/x5,225/x6,270/x7,315/x8}
{\node[vert] (\y) at (\x:1.5){};}

\foreach \x/\y in {0/v_1,45/v_2,90/v_3,135/v_4,180/v_5,225/v_6,270/v_7,315/v_8}
{\node[ver] at (\x:1){\scriptsize{$\y$}};}
\end{scope} 

\begin{scope}[shift = {(-3, -1.5)}, rotate=0 ]
\foreach \x/\y in {0/y2,45/y3,90/y4,135/y5,180/y6,225/y7,270/y8,315/y1}
{\node[vert] (\y) at (\x:1.5){};}

\foreach \x/\y in {0/v_9,45/v_{10},90/v_{11},135/v_{12},180/v_{13},225/v_{14},270/v_{15},315/v_{16}}
{\node[ver] at (\x:1)
{\scriptsize{$\y$}};}
\end{scope} 

\begin{scope}[shift = {(3, -1.5)}]
\foreach \x/\y in {0/z1,45/z2,90/z3,135/z4,180/z5,225/z6,270/z7,315/z8}
{\node[vert] (\y) at (\x:1.5){};}

\foreach \x/\y in {0/v_{17},45/v_{18},90/v_{19},135/v_{20},180/v_{21},225/v_{22},270/v_{23},315/v_{24}}
{\node[ver] at (\x:1){\scriptsize{$\y$}};}
\end{scope} 

\foreach \x/\y in
{1/1',1'/2,2/2',2'/3,3/3',3'/4,4/4',4'/5,5/5',5'/6,6/6',6'/7,7/7',7'/8,8/8',8'/9,9/9',9'/10,10/10',10'/11,11/11',11'/12,12/12',12'/13,13/13',13'/14,14/14',14'/15,15/15',15'/16,16/16',16'/17,17/17',17'/18,18/18',18'/1}
{\path[edge, dotted] (\x) -- (\y);}

\foreach \x/\y in
{x1/x2,x2/x3,x3/x4,x4/x5,x5/x6,x6/x7,x7/x8,x8/x1}
{\path[edge] (\x) -- (\y);}

\foreach \x/\y in {x1/x2,x3/x4,x5/x6,x7/x8}
{{\draw [line width=3pt, line cap=round, dash pattern=on 0pt off 1.3\pgflinewidth]  (\x) -- (\y);}}

\foreach \x/\y in
{y1/y2,y2/y3,y3/y4,y4/y5,y5/y6,y6/y7,y7/y8,y8/y1}
{\path[edge] (\x) -- (\y);}

\foreach \x/\y in {y1/y2,y3/y4,y5/y6,y7/y8}
{{\draw [line width=3pt, line cap=round, dash pattern=on 0pt off 1.3\pgflinewidth]  (\x) -- (\y);}}

\foreach \x/\y in
{z1/z2,z2/z3,z3/z4,z4/z5,z5/z6,z6/z7,z7/z8,z8/z1}
{\path[edge] (\x) -- (\y);}

\foreach \x/\y in {z1/z2,z3/z4,z5/z6,z7/z8}
{{\draw [line width=3pt, line cap=round, dash pattern=on 0pt off 1.3\pgflinewidth]  (\x) -- (\y);}}

\path[edge,dashed] (x6) -- (y3);
\path[edge,dashed] (y2) -- (z5);
\path[edge,dashed] (z4) -- (x7);

\foreach \x/\y in {x8/2,x1/3,x2/4,x3/5,x4/6,x5/7,z3/1,z2/18,z1/17,z8/16,z7/15,z6/14,y4/8,y5/9,y6/10,y7/11,y8/12,y1/13}
{\path[edge, dashed] (\x) -- (\y);}

\begin{scope} [scale=1, shift = {(-5, -8)}]
\foreach \x/\y/\z in {1/-0.5/0,5/-0.5/1,9/-0.5/2}
{\node[ver] () at (\x,\y){$\z$};}
\path[edge] (0,0) -- (2,0);
\path[edge,dashed] (4,0) -- (6,0);

\path[edge] (8,0) -- (10,0);
\draw [line width=3pt, line cap=round, dash pattern=on 0pt off 1.3\pgflinewidth]  (8,0) -- (10,0);

\end{scope}


 \end{tikzpicture}

 \caption{Embedding on $\#_3 \mathbb{RP}^2$ of gem representing $\#_{3} \mathbb{RP}^2$ of type $(6^2,8)$.}  \label{fig:2}
\end{figure}
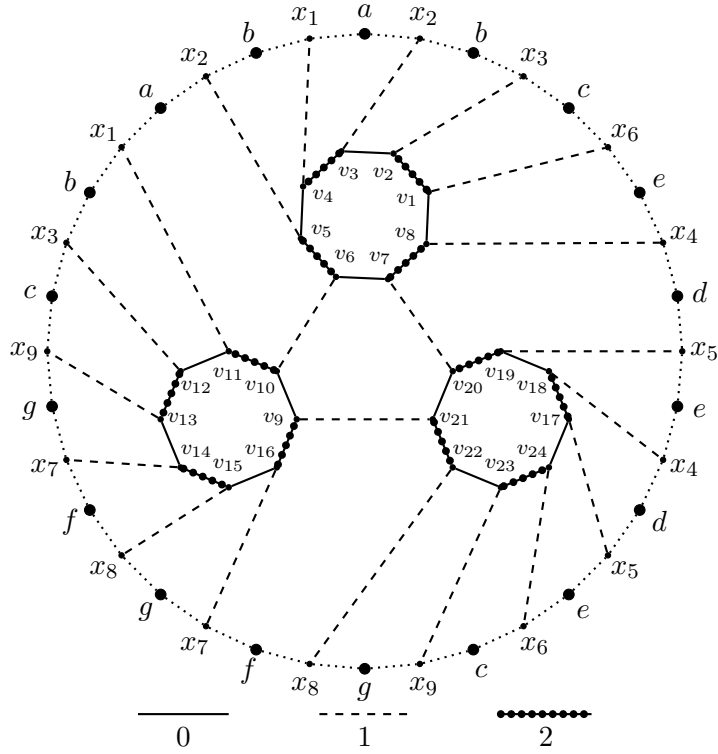

\noindent \textbf{Figure \ref{fig:3}.} In this diagram, the surface features three $0$-cells labeled $a$, $b$, and $c$, five $1$-cells, and one $2$-cell. Consequently, the Euler characteristic of the surface is $-1$. The $\{0,1\}$-colored two cycles $A_1, A_2$ bound the two hexagonal faces $v_1v_2v_3v_4v_5v_6$ and $v_7v_8v_9v_{10}v_{11}v_{12}$, respectively. The $\{1,2\}$-colored two $6$-cycles $B_1, B_2$ bound the two faces $v_1v_2x_3v_4v_3x_4v_8v_7x_5$ and $v_5v_6v_9v_{10}x_2v_{12}v_{11}x_1$, containing $0$-cells in interiors $c$ and $a$, respectively. The $\{0,2\}$-colored $12$-cycle bounds the face $v_1v_6v_9v_8x_4v_3v_2x_3v_4v_5x_1v_{11}v_{10}x_2v_{12}v_7x_5$ containing the $0$-cell $b$ in its interior. Thus, the 3-regular colored graph in Figure \ref{fig:3} is a semi-equivelar gem embedded regularly on the surface with Euler characteristic $-1$ of type $[(6^2,12);12]$.

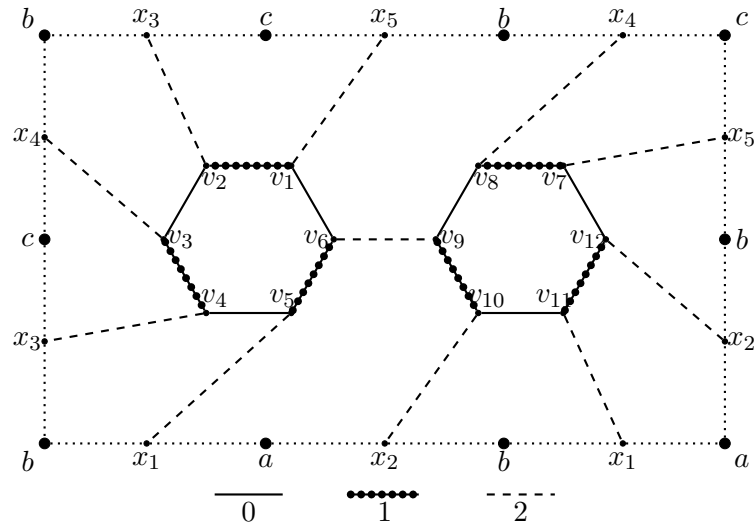
\begin{figure}[h!]
\tikzstyle{ver}=[]
\tikzstyle{vert}=[circle, draw, fill=black!100, inner sep=0pt, minimum width=2pt]
\tikzstyle{verti}=[circle, draw, fill=black!100, inner sep=0pt, minimum width=4pt]
\tikzstyle{edge} = [draw,thick,-]
    \centering
\begin{tikzpicture}[scale=0.45]

\begin{scope}
\foreach \x/\y/\z in
{10/6/1,3.5/6/2,-3.5/6/3,-10/6/4,-10/0/5,-10/-6/6,-3.5/-6/7,3.5/-6/8,10/-6/9,10/0/10}
{\node[verti] (\z) at (\x,\y){};}

\foreach \x/\y/\z in
{10.5/6.5/c,3.5/6.5/b,-3.5/6.5/c,-10.5/6.5/b,-10.5/0/c,-10.5/-6.5/b,-3.5/-6.5/a,3.5/-6.5/b,10.5/-6.5/a,10.5/0/b}
{\node[ver] (\z) at (\x,\y){$\z$};}

\foreach \x/\y/\z in
{7/6/1',0/6/2',-7/6/3',-10/3/4',-10/-3/5',-7/-6/6',0/-6/7',7/-6/8',10/-3/9',10/3/10'}
{\node[vert] (\z) at (\x,\y){};}

\foreach \x/\y/\z in
{7/6.5/x_4,0/6.5/x_5,-7/6.5/x_3,-10.5/3/x_4,-10.5/-3/x_3,-7/-6.5/x_1,0/-6.5/x_2,7/-6.5/x_1,10.5/-3/x_2,10.5/3/x_5}
{\node[ver] (\z) at (\x,\y){$\z$};}

\end{scope}

\begin{scope}[shift={(-4,0)}]
\foreach \x/\y in {60/x1,120/x2,180/x3,240/x4,300/x5,360/x6}
{\node[vert] (\y) at (\x:2.5){};}

\foreach \x/\y in {60/v_1,120/v_2,180/v_3,240/v_4,300/v_5,360/v_6}
{\node[ver] at (\x:2){$\y$};}
\end{scope}

\begin{scope}[shift={(4,0)}]
\foreach \x/\y in {60/y1,120/y2,180/y3,240/y4,300/y5,360/y6}
{\node[vert] (\y) at (\x:2.5){};}

\foreach \x/\y in {60/v_7,120/v_8,180/v_9,240/v_{10},300/v_{11},360/v_{12}}
{\node[ver] at (\x:2){$\y$};}
\end{scope}

\foreach \x/\y in {1/1',1'/2,2/2',2'/3,3/3',3'/4,4/4',4'/5,5/5',5'/6,6/6',6'/7,7/7',7'/8,8/8',8'/9,9/9',9'/10,10/10',10'/1}
{\path[edge,dotted] (\x) -- (\y);}

\foreach \x/\y in {x1/x2,x2/x3,x3/x4,x4/x5,x5/x6,x6/x1}
{\path[edge] (\x) -- (\y);}

\foreach \x/\y in {x1/x2,x3/x4,x5/x6}
{{\draw [line width=3pt, line cap=round, dash pattern=on 0pt off 1.3\pgflinewidth]  (\x) -- (\y);}}

\foreach \x/\y in {y1/y2,y2/y3,y3/y4,y4/y5,y5/y6,y6/y1}
{\path[edge] (\x) -- (\y);}

\foreach \x/\y in {y1/y2,y3/y4,y5/y6}
{{\draw [line width=3pt, line cap=round, dash pattern=on 0pt off 1.3\pgflinewidth]  (\x) -- (\y);}}

\foreach \x/\y in {x6/y3,y2/1',y1/10',y6/9',y5/8',y4/7',x1/2',x2/3',x3/4',x4/5',x5/6'}
{\path[edge,dashed] (\x) -- (\y);}

\begin{scope} [scale=1, shift = {(-5, -7.5)}]
\foreach \x/\y/\z in {1/-0.5/0,5/-0.5/1,9/-0.5/2}
{\node[ver] () at (\x,\y){$\z$};}
\path[edge] (0,0) -- (2,0);
\path[edge] (4,0) -- (6,0);
\draw [line width=3pt, line cap=round, dash pattern=on 0pt off 1.3\pgflinewidth]  (4,0) -- (6,0);
\path[edge, dashed] (8,0) -- (10,0);

\end{scope}  

 \end{tikzpicture}

 \caption{Embedding on $\#_3 \mathbb{RP}^2$ of gem representing $\#_{3} \mathbb{RP}^2$ of type $(6^2,12)$.}  \label{fig:3}
\end{figure}

\noindent \textbf{Figure \ref{fig:4}.} In this illustration, the surface contains two $0$-cells labeled $a$ and $b$, four $1$-cells, and one $2$-cell. Therefore, the Euler characteristic of the surface is $-1$. The $\{0,1\}$-colored two cycles $A_1, A_2$ bound the two faces $v_1v_2v_3v_4v_5v_6v_7v_8v_9v_{10}$ and $v_{11}v_{12}v_{13}v_{14}v_{15}v_{16}v_{17}v_{18}v_{19}v_{20}$, respectively. The $\{1,2\}$-colored two $10$-cycles $B_1, B_2$ bound the faces $v_1v_2x_2v_4v_3x_1v_5v_6x_5v_{11}v_{20}x_3$ $v_{14}v_{15}$ and $v_9v_{10}v_{16}v_{17}x_7v_{8}v_{7}x_6v_{12}v_{13}x_4v_{19}v_{18}x_8$, containing in their interiors, $0$-cells $a$ and $b$, respectively. The $\{0,2\}$-colored five $4$-cycles $C_1,C_2,C_3,C_4,C_5$ bound the faces $v_1v_{10}v_{16}v_{15}, v_3x_1$
$v_5v_4x_2v_2, v_6x_5v_{11}v_{12}x_6v_7, v_8x_7v_{17}v_{18}x_8v_9$, and $v_{13}x_4v_{19}v_{20}x_3v_{14}$, respectively. Thus, the 3-regular colored graph in Figure \ref{fig:4} is a semi-equivelar gem of type $[(10^2,4);20]$ embedded regularly on the surface $\#_3 \mathbb{RP}^2$.

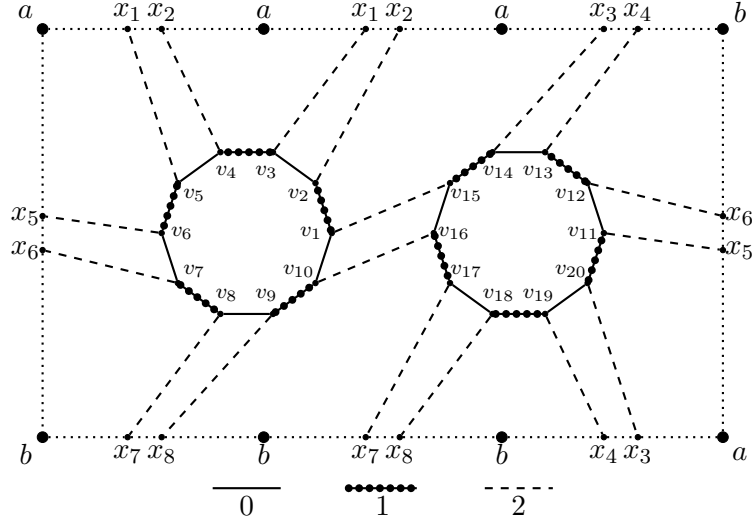
\begin{figure}[h!]
\tikzstyle{ver}=[]
\tikzstyle{vert}=[circle, draw, fill=black!100, inner sep=0pt, minimum width=2pt]
\tikzstyle{verti}=[circle, draw, fill=black!100, inner sep=0pt, minimum width=4pt]
\tikzstyle{edge} = [draw,thick,-]
    \centering
\begin{tikzpicture}[scale=0.45]

\begin{scope}
\foreach \x/\y/\z in
{10/6/1,3.5/6/2,-3.5/6/3,-10/6/4,-10/-6/5,-3.5/-6/6,3.5/-6/7,10/-6/8}
{\node[verti] (\z) at (\x,\y){};}

\foreach \x/\y/\z in
{10.5/6.5/b,3.5/6.5/a,-3.5/6.5/a,-10.5/6.5/a,-10.5/-6.5/b,-3.5/-6.5/b,3.5/-6.5/b,10.5/-6.5/a}
{\node[ver] (\z) at (\x,\y){$\z$};}

\foreach \x/\y/\z in
{7.5/6/1',6.5/6/2',0.5/6/3',-0.5/6/4',-6.5/6/5',-7.5/6/6',-10/0.5/7',-10/-0.5/8',-7.5/-6/9',-6.5/-6/10',-0.5/-6/11',0.5/-6/12',6.5/-6/13',7.5/-6/14',10/-0.5/15',10/0.5/16'}
{\node[vert] (\z) at (\x,\y){};}

\foreach \x/\y/\z in
{7.5/6.5/x_4,6.5/6.5/x_3,0.5/6.5/x_2,-0.5/6.5/x_1,-6.5/6.5/x_2,-7.5/6.5/x_1,-10.5/0.5/x_5,-10.5/-0.5/x_6,-7.5/-6.5/x_7,-6.5/-6.5/x_8,-0.5/-6.5/x_7,0.5/-6.5/x_8,6.5/-6.5/x_4,7.5/-6.5/x_{3},10.5/-0.5/x_{5},10.5/0.5/x_{6}}
{\node[ver] (\z) at (\x,\y){$\z$};}

\end{scope}

\begin{scope}[shift={(-4,0)}]
\foreach \x/\y in {0/x1,36/x2,72/x3,108/x4,144/x5,180/x6,216/x7,252/x8,288/x9,324/x10}
{\node[vert] (\y) at (\x:2.5){};}

\foreach \x/\y in {0/v_1,36/v_2,72/v_3,108/v_4,144/v_5,180/v_6,216/v_7,252/v_8,288/v_9,324/v_{10}}
{\node[ver] at (\x:1.9){\scriptsize{$\y$}};}
\end{scope}

\begin{scope}[shift={(4,0)}]
\foreach \x/\y in {0/y1,36/y2,72/y3,108/y4,144/y5,180/y6,216/y7,252/y8,288/y9,324/y10}
{\node[vert] (\y) at (\x:2.5){};}

\foreach \x/\y in {0/v_{11},36/v_{12},72/v_{13},108/v_{14},144/v_{15},180/v_{16},216/v_{17},252/v_{18},288/v_{19},324/v_{20}}
{\node[ver] at (\x:1.9){\scriptsize{$\y$}};}
\end{scope}

\foreach \x/\y in {1/1',1'/2',2'/2,2/3',3'/4',4'/3,3/5',5'/6',6'/4,4/7',7'/8',8'/5,5/9',9'/10',10'/6,6/11',11'/12',12'/7,7/13',13'/14',14'/8,8/15',15'/16',16'/1}
{\path[edge,dotted] (\x) -- (\y);}

\foreach \x/\y in {x1/x2,x2/x3,x3/x4,x4/x5,x5/x6,x6/x7,x7/x8,x8/x9,x9/x10,x10/x1}
{\path[edge] (\x) -- (\y);}

\foreach \x/\y in {x1/x2,x3/x4,x5/x6,x7/x8,x9/x10}
{{\draw [line width=3pt, line cap=round, dash pattern=on 0pt off 1.3\pgflinewidth]  (\x) -- (\y);}}

\foreach \x/\y in {y1/y2,y2/y3,y3/y4,y4/y5,y5/y6,y6/y7,y7/y8,y8/y9,y9/y10,y10/y1}
{\path[edge] (\x) -- (\y);}

\foreach \x/\y in {y2/y3,y4/y5,y6/y7,y8/y9,y10/y1}
{{\draw [line width=3pt, line cap=round, dash pattern=on 0pt off 1.3\pgflinewidth]  (\x) -- (\y);}}

\foreach \x/\y in {x10/y6,x1/y5,x2/3',x3/4',x4/5',x5/6',x6/7',x7/8',x8/9',x9/10',y7/11',y8/12',y9/13',y10/14',y1/15',y2/16',y3/1',y4/2'}
{\path[edge,dashed] (\x) -- (\y);}

\begin{scope} [scale=1, shift = {(-5, -7.5)}]
\foreach \x/\y/\z in {1/-0.5/0,5/-0.5/1,9/-0.5/2}
{\node[ver] () at (\x,\y){$\z$};}
\path[edge] (0,0) -- (2,0);
\path[edge] (4,0) -- (6,0);
\draw [line width=3pt, line cap=round, dash pattern=on 0pt off 1.3\pgflinewidth]  (4,0) -- (6,0);
\path[edge, dashed] (8,0) -- (10,0);

\end{scope}  

 \end{tikzpicture}

 \caption{Embedding on $\#_3 \mathbb{RP}^2$ of gem representing $\#_{3} \mathbb{RP}^2$ of type $(10^2,4)$.}  \label{fig:4}
\end{figure}

\noindent \textbf{Figure \ref{fig:5}.} In this diagram, the surface is composed of only one $0$-cell labeled $a$, three $1$-cells, and one $2$-cell. So, the Euler characteristic of the surface is $-1$. The $\{0,1\}$-colored cycle bounds the inner face $v_1v_2v_3v_4v_5v_6v_7v_8v_9v_{10}v_{11}v_{12}$. The $\{1,2\}$-colored $12$-cycle bounds the face $v_1x_2v_3v_2x_1v_4v_5x_6v_7v_6x_5v_8v_9x_4v_{11}v_{10}x_3v_{12}$, and the $0$-cell $a$ lies in the interior of this face. The $\{0,2\}$-colored three $4$-cycles $C_1, C_2, C_3$ bound faces $v_1x_2v_3v_4x_1v_2,  v_9x_4v_{11}v_{12}x_3v_{10},$ and $ v_1x_2v_3v_4x_1v_2$, respectively. Thus, the 3-regular colored graph in Figure \ref{fig:5} is a semi-equivelar gem of type $[(12^2,4);12]$ embedded regularly on the surface $\#_3 \mathbb{RP}^2$.

\begin{figure}[h!]
\tikzstyle{ver}=[]
\tikzstyle{vert}=[circle, draw, fill=black!100, inner sep=0pt, minimum width=2pt]
\tikzstyle{verti}=[circle, draw, fill=black!100, inner sep=0pt, minimum width=4pt]
\tikzstyle{edge} = [draw,thick,-]
    \centering
\begin{tikzpicture}[scale=0.6]
\begin{scope}

\foreach \x/\y in {0/1,30/2,60/3,90/4,120/5,150/6,180/7,210/8,240/9,270/10,300/11,330/12}
{\node[vert] (\y) at (\x:3){};}

\foreach \x/\y in {0/v_1,30/v_2,60/v_3,90/v_4,120/v_5,150/v_6,180/v_7,210/v_8,240/v_9,270/v_{10},300/v_{11},330/v_{12}}
{\node[ver] at (\x:2.5){\small{$\y$}};}

\foreach \x/\y in {10/1',20/2',70/3',80/4',130/5',140/6',190/7',200/8',250/9',260/10',310/11',320/12'}
{\node[vert] (\y) at (\x:6){};}

\foreach \x/\y in {45/1'',105/2'',165/3'',225/4'',285/5'',345/6''}
{\node[verti] (\y) at (\x:6){};}

\foreach \x/\y in {10/x_2,20/x_1,70/x_2,80/x_1,130/x_6,140/x_5,190/x_6,200/x_5,250/x_4,260/x_3,310/x_4,320/x_3}
{\node[ver] (\y) at (\x:6.5){$\y$};}

\foreach \x/\y in {45/a,105/a,165/a,225/a,285/a,345/a}
{\node[ver] (\y) at (\x:6.5){$\y$};}

\foreach \x/\y in {1/1',2/2',3/3',4/4',5/5',6/6',7/7',8/8',9/9',10/10',11/11',12/12'}
{\path[edge,dashed] (\x) -- (\y);}

\foreach \x/\y in {1/2,2/3,3/4,4/5,5/6,6/7,7/8,8/9,9/10,10/11,11/12,12/1}
{\path[edge] (\x) -- (\y);}

\foreach \x/\y in {2/3,4/5,6/7,8/9,10/11,12/1}
{{\draw [line width=3pt, line cap=round, dash pattern=on 0pt off 1.3\pgflinewidth]  (\x) -- (\y);}}

\foreach \x/\y in {1'/2',2'/1'',1''/3',3'/4',4'/2'',2''/5',5'/6',6'/3'',3''/7',7'/8',8'/4'',4''/9',9'/10',10'/5'',5''/11',11'/12',12'/6'',6''/1'}
{\path[edge,dotted] (\x) -- (\y);}

\end{scope}

\begin{scope} [scale=1, shift = {(-5, -7.2)}]
\foreach \x/\y/\z in {1/-0.5/0,5/-0.5/1,9/-0.5/2}
{\node[ver] () at (\x,\y){$\z$};}
\path[edge] (0,0) -- (2,0);
\path[edge] (4,0) -- (6,0);
\draw [line width=3pt, line cap=round, dash pattern=on 0pt off 1.3\pgflinewidth]  (4,0) -- (6,0);
\path[edge, dashed] (8,0) -- (10,0);

\end{scope}  

    \end{tikzpicture}

 \caption{Embedding on $\#_3 \mathbb{RP}^2$ of gem representing $\#_{3} \mathbb{RP}^2$ of type $(12^2,4)$.} \label{fig:5}
\end{figure}
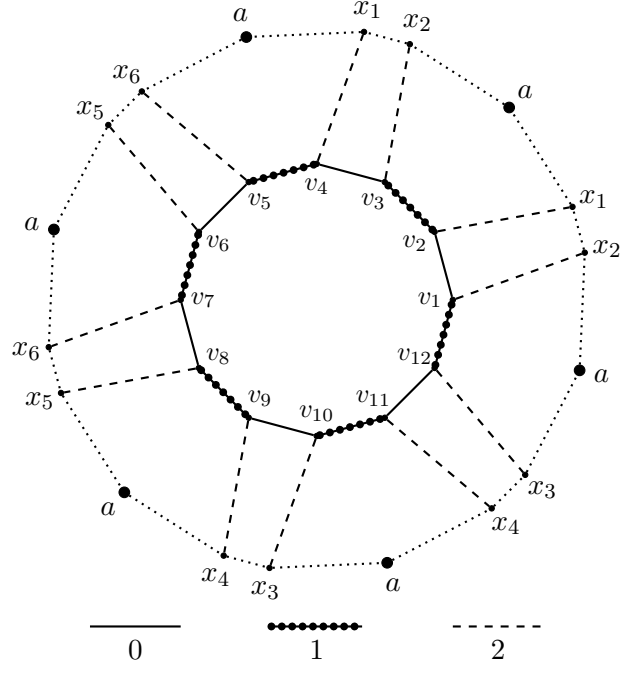

\noindent \textbf{Figure \ref{fig:6}.} The surface in this figure exhibits six $0$-cells labeled $a,b,c,d,e$, and $f$, eight $1$-cells, and one $2$-cell. Hence, the Euler characteristic of the surface is $-1$. The $\{1,2\}$-colored fourteen $6$-cycles $A_1,A_2,\dots,A_{14}$ bound the hexagonal faces $v_1v_2v_3v_4v_5v_6,  v_7v_8v_9v_{10}v_{11}v_{12},\dots,$ $v_{79}v_{80}v_{81}v_{82}v_{83}v_{84}$, respectively. The $\{0,2\}$-colored six $14$-cycles $B_1,B_2,B_3,B_4,B_5$, $B_6$ bound the faces $v_1x_{16}v_{47}v_{46}x_{12}v_{83}v_{82}v_{78}v_{77}v_{71}v_{70}v_{66}v_{65}x_{10}v_{9}v_{10}v_{6}$ containing $d$, $v_8x_{9}v_{64}v_{63}v_{22}v_{21}v_{13}$ $v_{18}v_{55}v_{60}x_5v_{19}v_{20}v_{14}v_{15}v_{7}$ containing $b$, $v_4x_{13}v_{79}v_{84}x_{11}v_{45}v_{44}v_{50}v_{49}v_{57}v_{56}v_{17}v_{16}v_{12}v_{11}v_{5}$ containing $a$, $v_{25}x_{2}v_{32}v_{31}v_{39}v_{38}x_4v_{27}v_{26}x_1v_{33}v_{34}v_{75}v_{76}v_{72}v_{67}v_{30}$ containing $c$, $v_{24}x_{6}v_{59}v_{58}v_{54}v_{53}x_7v_{42}v_{37}$ $x_3v_{28}v_{29}v_{68}v_{69}v_{61}v_{62}v_{23}$ containing $e$, and $v_2x_{15}v_{48}v_{43}v_{51}v_{52}x_8v_{41}v_{40}v_{36}v_{35}v_{74}v_{73}v_{81}v_{80}x_{14}v_{3}$ containing $f$, respectively. The $\{0,1\}$-colored twenty-one $4$-cycles $C_1, C_2,\dots,C_{21}$ bound faces $v_6v_{10}v_{11}v_{5}, v_1x_{16}v_{47}v_{48}x_{15}v_{2},\dots,v_8x_9v_{64}v_{65}x_{10}v_{9}$, respectively. Thus, the 3-regular colored graph in Figure \ref{fig:6} is a semi-equivelar gem of type $[(4,6,14);84]$ embedded regularly on the surface $\#_3 \mathbb{RP}^2$.

\begin{figure}[h!]
\tikzstyle{ver}=[]
\tikzstyle{vert}=[circle, draw, fill=black!100, inner sep=0pt, minimum width=2pt]
\tikzstyle{verti}=[circle, draw, fill=black!100, inner sep=0pt, minimum width=4pt]
\tikzstyle{edge} = [draw,thick,-]
    \centering
\begin{tikzpicture}[scale=0.5]

\begin{scope}[shift={(-12,2.5)},rotate=-30]
\foreach \x/\y in {60/a1,120/a2,180/a3,240/a4,300/a5,360/a6}
{\node[vert] (\y) at (\x:1.2){};}

\foreach \x/\y in {60/v_1,120/v_2,180/v_3,240/v_4,300/v_5,360/v_6}
{\node[ver] at (\x:0.75){\tiny{$\y$}};}

\foreach \x/\y in {a1/a2,a2/a3,a3/a4,a4/a5,a5/a6,a6/a1}
{\path[edge] (\x) -- (\y);}
\foreach \x/\y in {a1/a6,a2/a3,a4/a5}
{{\draw [line width=3pt, line cap=round, dash pattern=on 0pt off 1.3\pgflinewidth]  (\x) -- (\y);}}
\end{scope}

\begin{scope}[shift={(-8,2.5)},rotate=-30]
\foreach \x/\y in {60/b1,120/b2,180/b3,240/b4,300/b5,360/b6}
{\node[vert] (\y) at (\x:1.2){};}

\foreach \x/\y in {60/v_7,120/v_8,180/v_9,240/v_{10},300/v_{11},360/v_{12}}
{\node[ver] at (\x:0.75){\tiny{$\y$}};}

\foreach \x/\y in {b1/b2,b2/b3,b3/b4,b4/b5,b5/b6,b6/b1}
{\path[edge] (\x) -- (\y);}
\foreach \x/\y in {b1/b2,b3/b4,b5/b6}
{{\draw [line width=3pt, line cap=round, dash pattern=on 0pt off 1.3\pgflinewidth]  (\x) -- (\y);}}
\end{scope}

\begin{scope}[shift={(-4,2.5)},rotate=-30]
\foreach \x/\y in {60/c1,120/c2,180/c3,240/c4,300/c5,360/c6}
{\node[vert] (\y) at (\x:1.2){};}

\foreach \x/\y in {60/v_{13},120/v_{14},180/v_{15},240/v_{16},300/v_{17},360/v_{18}}
{\node[ver] at (\x:0.75){\tiny{$\y$}};}

\foreach \x/\y in {c1/c2,c2/c3,c3/c4,c4/c5,c5/c6,c6/c1}
{\path[edge] (\x) -- (\y);}
\foreach \x/\y in {c1/c6,c2/c3,c4/c5}
{{\draw [line width=3pt, line cap=round, dash pattern=on 0pt off 1.3\pgflinewidth]  (\x) -- (\y);}}
\end{scope}

\begin{scope}[shift={(0,2.5)},rotate=-30]
\foreach \x/\y in {60/d1,120/d2,180/d3,240/d4,300/d5,360/d6}
{\node[vert] (\y) at (\x:1.2){};}

\foreach \x/\y in {60/v_{19},120/v_{20},180/v_{21},240/v_{22},300/v_{23},360/v_{24}}
{\node[ver] at (\x:0.75){\tiny{$\y$}};}

\foreach \x/\y in {d1/d2,d2/d3,d3/d4,d4/d5,d5/d6,d6/d1}
{\path[edge] (\x) -- (\y);}
\foreach \x/\y in {d1/d2,d3/d4,d5/d6}
{{\draw [line width=3pt, line cap=round, dash pattern=on 0pt off 1.3\pgflinewidth]  (\x) -- (\y);}}
\end{scope}

\begin{scope}[shift={(4,2.5)},rotate=-30]
\foreach \x/\y in {60/e1,120/e2,180/e3,240/e4,300/e5,360/e6}
{\node[vert] (\y) at (\x:1.2){};}

\foreach \x/\y in {60/v_{25},120/v_{26},180/v_{27},240/v_{28},300/v_{29},360/v_{30}}
{\node[ver] at (\x:0.75){\tiny{$\y$}};}

\foreach \x/\y in {e1/e2,e2/e3,e3/e4,e4/e5,e5/e6,e6/e1}
{\path[edge] (\x) -- (\y);}
\foreach \x/\y in {e1/e6,e2/e3,e4/e5}
{{\draw [line width=3pt, line cap=round, dash pattern=on 0pt off 1.3\pgflinewidth]  (\x) -- (\y);}}
\end{scope}

\begin{scope}[shift={(8,2.5)},rotate=-30]
\foreach \x/\y in {60/f1,120/f2,180/f3,240/f4,300/f5,360/f6}
{\node[vert] (\y) at (\x:1.2){};}

\foreach \x/\y in {60/v_{31},120/v_{32},180/v_{33},240/v_{34},300/v_{35},360/v_{36}}
{\node[ver] at (\x:0.75){\tiny{$\y$}};}

\foreach \x/\y in {f1/f2,f2/f3,f3/f4,f4/f5,f5/f6,f6/f1}
{\path[edge] (\x) -- (\y);}
\foreach \x/\y in {f1/f2,f3/f4,f5/f6}
{{\draw [line width=3pt, line cap=round, dash pattern=on 0pt off 1.3\pgflinewidth]  (\x) -- (\y);}}
\end{scope}

\begin{scope}[shift={(12,2.5)},rotate=-30]
\foreach \x/\y in {60/g1,120/g2,180/g3,240/g4,300/g5,360/g6}
{\node[vert] (\y) at (\x:1.2){};}

\foreach \x/\y in {60/v_{37},120/v_{38},180/v_{39},240/v_{40},300/v_{41},360/v_{42}}
{\node[ver] at (\x:0.75){\tiny{$\y$}};}

\foreach \x/\y in {g1/g2,g2/g3,g3/g4,g4/g5,g5/g6,g6/g1}
{\path[edge] (\x) -- (\y);}
\foreach \x/\y in {g1/g6,g2/g3,g4/g5}
{{\draw [line width=3pt, line cap=round, dash pattern=on 0pt off 1.3\pgflinewidth]  (\x) -- (\y);}}
\end{scope}

\begin{scope}[shift={(-12,-2.5)},rotate=-30]
\foreach \x/\y in {60/h1,120/h2,180/h3,240/h4,300/h5,360/h6}
{\node[vert] (\y) at (\x:1.2){};}

\foreach \x/\y in {60/v_{43},120/v_{44},180/v_{45},240/v_{46},300/v_{47},360/v_{48}}
{\node[ver] at (\x:0.75){\tiny{$\y$}};}

\foreach \x/\y in {h1/h2,h2/h3,h3/h4,h4/h5,h5/h6,h6/h1}
{\path[edge] (\x) -- (\y);}
\foreach \x/\y in {h1/h6,h2/h3,h4/h5}
{{\draw [line width=3pt, line cap=round, dash pattern=on 0pt off 1.3\pgflinewidth]  (\x) -- (\y);}}
\end{scope}

\begin{scope}[shift={(-8,-2.5)},rotate=-30]
\foreach \x/\y in {60/i1,120/i2,180/i3,240/i4,300/i5,360/i6}
{\node[vert] (\y) at (\x:1.2){};}

\foreach \x/\y in {60/v_{49},120/v_{50},180/v_{51},240/v_{52},300/v_{53},360/v_{54}}
{\node[ver] at (\x:0.75){\tiny{$\y$}};}

\foreach \x/\y in {i1/i2,i2/i3,i3/i4,i4/i5,i5/i6,i6/i1}
{\path[edge] (\x) -- (\y);}
\foreach \x/\y in {i1/i2,i3/i4,i5/i6}
{{\draw [line width=3pt, line cap=round, dash pattern=on 0pt off 1.3\pgflinewidth]  (\x) -- (\y);}}
\end{scope}

\begin{scope}[shift={(-4,-2.5)},rotate=-30]
\foreach \x/\y in {60/j1,120/j2,180/j3,240/j4,300/j5,360/j6}
{\node[vert] (\y) at (\x:1.2){};}

\foreach \x/\y in {60/v_{55},120/v_{56},180/v_{57},240/v_{58},300/v_{59},360/v_{60}}
{\node[ver] at (\x:0.75){\tiny{$\y$}};}

\foreach \x/\y in {j1/j2,j2/j3,j3/j4,j4/j5,j5/j6,j6/j1}
{\path[edge] (\x) -- (\y);}
\foreach \x/\y in {j1/j6,j2/j3,j4/j5}
{{\draw [line width=3pt, line cap=round, dash pattern=on 0pt off 1.3\pgflinewidth]  (\x) -- (\y);}}
\end{scope}

\begin{scope}[shift={(0,-2.5)},rotate=-30]
\foreach \x/\y in {60/k1,120/k2,180/k3,240/k4,300/k5,360/k6}
{\node[vert] (\y) at (\x:1.2){};}

\foreach \x/\y in {60/v_{61},120/v_{62},180/v_{63},240/v_{64},300/v_{65},360/v_{66}}
{\node[ver] at (\x:0.75){\tiny{$\y$}};}

\foreach \x/\y in {k1/k2,k2/k3,k3/k4,k4/k5,k5/k6,k6/k1}
{\path[edge] (\x) -- (\y);}
\foreach \x/\y in {k1/k2,k3/k4,k5/k6}
{{\draw [line width=3pt, line cap=round, dash pattern=on 0pt off 1.3\pgflinewidth]  (\x) -- (\y);}}
\end{scope}

\begin{scope}[shift={(4,-2.5)},rotate=-30]
\foreach \x/\y in {60/l1,120/l2,180/l3,240/l4,300/l5,360/l6}
{\node[vert] (\y) at (\x:1.2){};}

\foreach \x/\y in {60/v_{67},120/v_{68},180/v_{69},240/v_{70},300/v_{71},360/v_{72}}
{\node[ver] at (\x:0.75){\tiny{$\y$}};}

\foreach \x/\y in {l1/l2,l2/l3,l3/l4,l4/l5,l5/l6,l6/l1}
{\path[edge] (\x) -- (\y);}
\foreach \x/\y in {l1/l6,l2/l3,l4/l5}
{{\draw [line width=3pt, line cap=round, dash pattern=on 0pt off 1.3\pgflinewidth]  (\x) -- (\y);}}
\end{scope}

\begin{scope}[shift={(8,-2.5)},rotate=-30]
\foreach \x/\y in {60/m1,120/m2,180/m3,240/m4,300/m5,360/m6}
{\node[vert] (\y) at (\x:1.2){};}

\foreach \x/\y in {60/v_{73},120/v_{74},180/v_{75},240/v_{76},300/v_{77},360/v_{78}}
{\node[ver] at (\x:0.75){\tiny{$\y$}};}

\foreach \x/\y in {m1/m2,m2/m3,m3/m4,m4/m5,m5/m6,m6/m1}
{\path[edge] (\x) -- (\y);}
\foreach \x/\y in {m1/m2,m3/m4,m5/m6}
{{\draw [line width=3pt, line cap=round, dash pattern=on 0pt off 1.3\pgflinewidth]  (\x) -- (\y);}}
\end{scope}

\begin{scope}[shift={(12,-2.5)},rotate=-30]
\foreach \x/\y in {60/n1,120/n2,180/n3,240/n4,300/n5,360/n6}
{\node[vert] (\y) at (\x:1.2){};}

\foreach \x/\y in {60/v_{79},120/v_{80},180/v_{81},240/v_{82},300/v_{83},360/v_{84}}
{\node[ver] at (\x:0.75){\tiny{$\y$}};}

\foreach \x/\y in {n1/n2,n2/n3,n3/n4,n4/n5,n5/n6,n6/n1}
{\path[edge] (\x) -- (\y);}
\foreach \x/\y in {n1/n6,n2/n3,n4/n5}
{{\draw [line width=3pt, line cap=round, dash pattern=on 0pt off 1.3\pgflinewidth]  (\x) -- (\y);}}
\end{scope}

\foreach \x/\y in {a6/b4,a5/b5,b1/c3,b6/c4,c1/d3,c2/d2,f1/g3,f6/g4,h1/i3,h2/i2,i1/j3,i6/j4,k1/l3,k6/l4,l6/m4,l5/m5,m1/n3,m6/n4,c5/j2,c6/j1,d5/k2,d4/k3,e5/l2,e6/l1,f5/m2,f4/m3}
{\path[edge,dashed] (\x) -- (\y);}

\foreach \x/\y/\z in {-12.5/6/x15, -11/6/x16, -9/6/x10,-8/6/x9, -1/6/x5, 0/6/x6, 2/6/x3,3/6/x4, 5/6/x1,6/6/x2,8/6/y1,9/6/y2, 12/6/y4,13/6/y3,14/2/x7,14/1/x8, 14/-2/x13,14/-1/x14,11.5/-6/x12,13/-6/x11,2/-6/y10,1/-6/y9,-2/-6/y5,-3/-6/y6,-6/-6/y7,-7/-6/y8,-11.5/-6/y16,-10/-6/y15,-14/3/y14,-14/2/y13,-14/-3/y12,-14/-2/y11}
{\node[vert] (\z) at (\x,\y){};}

\foreach \x/\y in {a2/x15, a1/x16,b2/x9,b3/x10, d1/x5,e3/x4,e2/x1,e1/x2,f3/y1,f2/y2,g2/y4,g1/y3,g6/x7, g5/x8, n1/x13,n2/x14,x12/n5,x11/n6,k4/y9,k5/y10,j6/y5,j5/y6,i4/y8,i5/y7,h5/y16,h6/y15,a3/y14,a4/y13,h3/y11,h4/y12}
{\path[edge,dashed] (\x) -- (\y);}

\draw[edge, dashed] plot [smooth,tension=0.7] coordinates{(d6) (1.3,3) (x6)};
\draw[edge, dashed] plot [smooth,tension=0.7] coordinates{(e4) (2.4,3) (x3)};

\path[edge,dotted] (-14,6) -- (14,6);
\path[edge,dotted] (-14,-6) -- (14,-6);
\path[edge,dotted] (-14,-6) -- (-14,6);
\path[edge,dotted] (14,-6) -- (14,6);

\foreach \x/\y/\z in {-12.5/6.5/x_{15}, -11/6.5/x_{16}, -9/6.5/x_{10},-8/6.5/x_9, -1/6.5/x_5, 0/6.5/x_6, 2/6.5/x_3,3/6.5/x_4, 5/6.5/x_1,6/6.5/x_2,8/6.5/x_1,9/6.5/x_2, 12/6.5/x_4,13/6.5/x_3,14.7/2/x_7,14.7/1/x_8, 14.7/-2/x_{13},14.7/-1/x_{14},11.5/-6.5/x_{12},13/-6.5/x_{11},2/-6.5/x_{10},1/-6.5/x_9,-2/-6.5/x_5,-3/-6.5/x_6,-6/-6.5/x_7,-7/-6.5/x_8,-11.5/-6.5/x_{16},-10/-6.5/x_{15},-14.8/3/x_{14},-14.8/2/x_{13},-14.8/-3/x_{12},-14.8/-2/x_{11}}
{\node[ver] () at (\x,\y){$\z$};}

\foreach \x/\y in
{-14/6,-10/6,-4/6,1/6,4/6,7/6,10.5/6,14/6,14/0,14/-6,6/-6,-14/0,-14/-6,-8.5/-6,-4.5/-6,-0.5/-6}
{\node[verti] () at (\x,\y){};}

\foreach \x/\y/\z in
{-14.5/6.5/f,-10/6.7/d,-4/6.7/b,1/6.7/e,4/6.7/c,7/6.7/c,10.5/6.7/c,14.4/6.4/e,14.4/0/f,14.5/-6.5/a,6/-6.7/d,-14.5/0/a,-14.5/-6.5/d,-8.5/-6.7/f,-4.5/-6.7/e,-0.5/-6.7/b}
{\node[ver] () at (\x,\y){$\z$};}

\begin{scope} [scale=1, shift = {(-5, -8)}]
\foreach \x/\y/\z in {1/-0.5/0,5/-0.5/1,9/-0.5/2}
{\node[ver] () at (\x,\y){$\z$};}
\path[edge,dashed] (0,0) -- (2,0);
\path[edge] (8,0) -- (10,0);
\draw [line width=3pt, line cap=round, dash pattern=on 0pt off 1.3\pgflinewidth]  (8,0) -- (10,0);
\path[edge] (4,0) -- (6,0);

\end{scope}  

\end{tikzpicture}
 \caption{Embedding on $\#_3 \mathbb{RP}^2$ of gem representing $\#_{3} \mathbb{RP}^2$ of type $(4,6,14)$.}  \label{fig:6}
\end{figure}
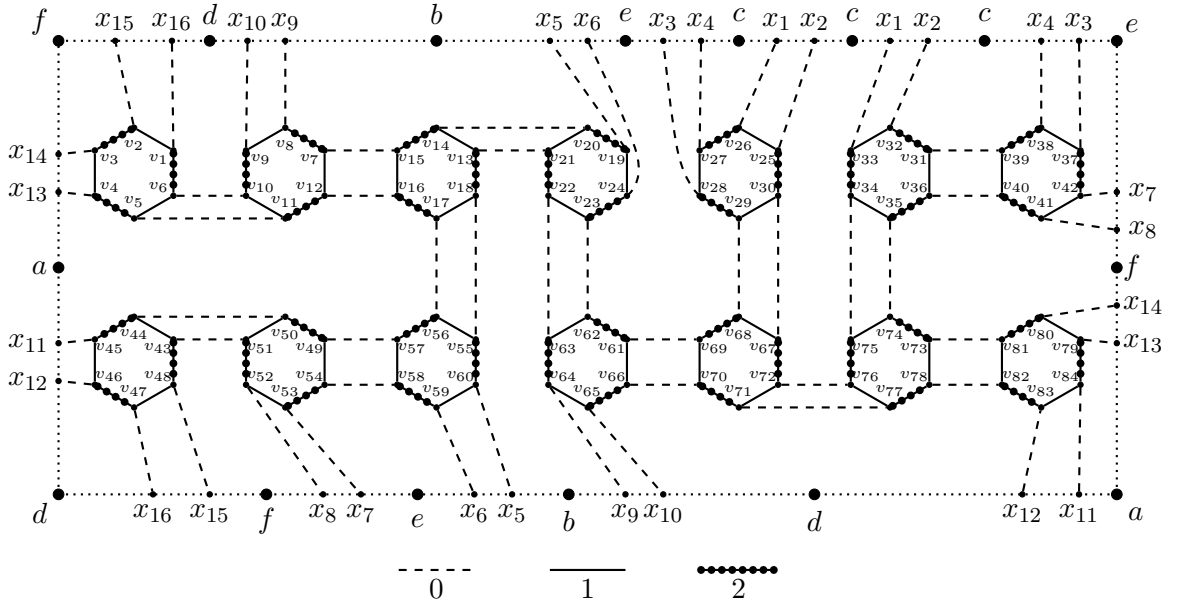

\noindent \textbf{Figure \ref{fig:7}.} In this diagram, the surface features two $0$-cells labeled $a$ and $b$, four $1$-cells, and one $2$-cell. Consequently, the Euler characteristic of the surface is $-1$. The $\{1,2\}$-colored eight $6$-cycles $A_1,A_2,\dots,A_8$ bound the hexagonal faces $v_1v_2v_3v_4v_5v_6, v_7v_8v_9v_{10}v_{11}v_{12},\dots, v_{43}v_{44}v_{45}$ $v_{46}v_{47}v_{48}$, respectively. The $\{0,2\}$-colored three $16$-cycles $B_1,B_2,B_3$ bound the faces $v_1v_{11}v_{12}$ $v_{17}v_{18}v_{24}v_{19}v_{45}v_{46}v_{38}v_{39}v_{34}v_{35}v_{25}v_{26}v_{6}$, $v_3x_3v_{31}v_{36}v_{30}v_{29}x_1v_{22}v_{23}v_{13}v_{14}x_7v_{48}v_{47}v_{37}v_{42}x_5v_{9}v_{10}v_{2}$ containing $a$, $v_4x_4v_{32}v_{33}v_{40}v_{41}x_6v_{8}v_{7}v_{16}v_{15}x_8v_{43}v_{44}v_{20}v_{21}x_2v_{28}v_{27}v_{5}$ containing $b$, respectively. The $\{0,1\}$-colored twelve $4$-cycles $C_1,C_2,\dots, C_{12}$ bound the faces $v_2v_{10}v_{11}v_{1}, v_3x_{3}v_{31}v_{32}x_{4}v_{4},$ $\dots, v_9x_{5}v_{42}v_{41}x_{6}v_{8}$, respectively. Thus, the 3-regular colored graph in Figure \ref{fig:7} is a semi-equivelar gem embedded regularly on the surface $\#_3 \mathbb{RP}^2$ of type $[(4,6,16);48]$.

\begin{figure}[h!]
\tikzstyle{ver}=[]
\tikzstyle{vert}=[circle, draw, fill=black!100, inner sep=0pt, minimum width=2pt]
\tikzstyle{verti}=[circle, draw, fill=black!100, inner sep=0pt, minimum width=4pt]
\tikzstyle{edge} = [draw,thick,-]
    \centering
\begin{tikzpicture}[scale=0.6]

\begin{scope}
\foreach \x/\y/\z in
{10/6/1,3.5/6/2,-3.5/6/3,-10/6/4,-10/-6/5,-3.5/-6/6,3.5/-6/7,10/-6/8}
{\node[verti] (\z) at (\x,\y){};}

\foreach \x/\y/\z in
{10.5/6.5/a,3.5/6.5/b,-3.5/6.5/a,-10.5/6.5/b,-10.5/-6.5/a,-3.5/-6.5/b,3.5/-6.5/a,10.5/-6.5/b}
{\node[ver] (\z) at (\x,\y){$\z$};}

\foreach \x/\y/\z in
{7.5/6/1',6.5/6/2',0.5/6/3',-0.5/6/4',-6.5/6/5',-7.5/6/6',-10/0.5/7',-10/-0.5/8',-7.5/-6/9',-6.5/-6/10',-0.5/-6/11',0.5/-6/12',6.5/-6/13',7.5/-6/14',10/-0.5/15',10/0.5/16'}
{\node[vert] (\z) at (\x,\y){};}

\foreach \x/\y/\z in
{7.5/6.5/x_7,6.5/6.5/x_8,0.5/6.5/x_6,-0.5/6.5/x_5,-6.5/6.5/x_3,-7.5/6.5/x_4,-10.5/0.5/x_2,-10.5/-0.5/x_1,-7.5/-6.5/x_3,-6.5/-6.5/x_4,-0.5/-6.5/x_6,0.5/-6.5/x_5,6.5/-6.5/x_7,7.5/-6.5/x_8,10.5/-0.5/x_{2},10.5/0.5/x_{1}}
{\node[ver] (\z) at (\x,\y){$\z$};}

\end{scope}

\begin{scope}[shift={(-6,2)},rotate=-30]
\foreach \x/\y in {0/p1,60/p2,120/p3,180/p4,240/p5,300/p6}
{\node[vert] (\y) at (\x:1.2){};}

\foreach \x/\y in {0/v_1,60/v_2,120/v_3,180/v_4,240/v_5,300/v_6}
{\node[ver] at (\x:0.75){\tiny{$\y$}};}
\end{scope}

\begin{scope}[shift={(-2,2)},rotate=-30]
\foreach \x/\y in {0/q1,60/q2,120/q3,180/q4,240/q5,300/q6}
{\node[vert] (\y) at (\x:1.2){};}

\foreach \x/\y in {0/v_7,60/v_8,120/v_9,180/v_{10},240/v_{11},300/v_{12}}
{\node[ver] at (\x:0.75){\tiny{$\y$}};}
\end{scope}

\begin{scope}[shift={(2,2)},rotate=-30]
\foreach \x/\y in {0/r1,60/r2,120/r3,180/r4,240/r5,300/r6}
{\node[vert] (\y) at (\x:1.2){};}

\foreach \x/\y in {0/v_{13},60/v_{14},120/v_{15},180/v_{16},240/v_{17},300/v_{18}}
{\node[ver] at (\x:0.75){\tiny{$\y$}};}
\end{scope}

\begin{scope}[shift={(6,2)},rotate=-30]
\foreach \x/\y in {0/s1,60/s2,120/s3,180/s4,240/s5,300/s6}
{\node[vert] (\y) at (\x:1.2){};}

\foreach \x/\y in {0/v_{19},60/v_{20},120/v_{21},180/v_{22},240/v_{23},300/v_{24}}
{\node[ver] at (\x:0.75){\tiny{$\y$}};}
\end{scope}

\begin{scope}[shift={(-6,-2)},rotate=-30]
\foreach \x/\y in {0/t1,60/t2,120/t3,180/t4,240/t5,300/t6}
{\node[vert] (\y) at (\x:1.2){};}

\foreach \x/\y in {0/v_{25},60/v_{26},120/v_{27},180/v_{28},240/v_{29},300/v_{30}}
{\node[ver] at (\x:0.75){\tiny{$\y$}};}
\end{scope}

\begin{scope}[shift={(-2,-2)},rotate=-30]
\foreach \x/\y in {0/u1,60/u2,120/u3,180/u4,240/u5,300/u6}
{\node[vert] (\y) at (\x:1.2){};}

\foreach \x/\y in {0/v_{31},60/v_{32},120/v_{33},180/v_{34},240/v_{35},300/v_{36}}
{\node[ver] at (\x:0.75){\tiny{$\y$}};}
\end{scope}

\begin{scope}[shift={(2,-2)},rotate=-30]
\foreach \x/\y in {0/v1,60/v2,120/v3,180/v4,240/v5,300/v6}
{\node[vert] (\y) at (\x:1.2){};}

\foreach \x/\y in {0/v_{37},60/v_{38},120/v_{39},180/v_{40},240/v_{41},300/v_{42}}
{\node[ver] at (\x:0.75){\tiny{$\y$}};}
\end{scope}

\begin{scope}[shift={(6,-2)},rotate=-30]
\foreach \x/\y in {0/w1,60/w2,120/w3,180/w4,240/w5,300/w6}
{\node[vert] (\y) at (\x:1.2){};}

\foreach \x/\y in {0/v_{43},60/v_{44},120/v_{45},180/v_{46},240/v_{47},300/v_{48}}
{\node[ver] at (\x:0.75){\tiny{$\y$}};}
\end{scope}

\foreach \x/\y in {1/1',1'/2',2'/2,2/3',3'/4',4'/3,3/5',5'/6',6'/4,4/7',7'/8',8'/5,5/9',9'/10',10'/6,6/11',11'/12',12'/7,7/13',13'/14',14'/8,8/15',15'/16',16'/1}
{\path[edge,dotted] (\x) -- (\y);}

\foreach \x/\y in {p1/p2,p2/p3,p3/p4,p4/p5,p5/p6,p6/p1,q1/q2,q2/q3,q3/q4,q4/q5,q5/q6,q6/q1,r1/r2,r2/r3,r3/r4,r4/r5,r5/r6,r6/r1,s1/s2,s2/s3,s3/s4,s4/s5,s5/s6,s6/s1,t1/t2,t2/t3,t3/t4,t4/t5,t5/t6,t6/t1,u1/u2,u2/u3,u3/u4,u4/u5,u5/u6,u6/u1,v1/v2,v2/v3,v3/v4,v4/v5,v5/v6,v6/v1,w1/w2,w2/w3,w3/w4,w4/w5,w5/w6,w6/w1}
{\path[edge] (\x) -- (\y);}

\foreach \x/\y in {p1/p2,p3/p4,p5/p6,s1/s2,s3/s4,s5/s6,u1/u2,u3/u4,u5/u6,v1/v2,v3/v4,v5/v6,q1/q6,q2/q3,q4/q5,r1/r6,r2/r3,r4/r5,t1/t6,t2/t3,t4/t5,w1/w6,w2/w3,w4/w5}
{{\draw [line width=3pt, line cap=round, dash pattern=on 0pt off 1.3\pgflinewidth]  (\x) -- (\y);}}

\foreach \x/\y in {p2/q4,p1/q5,q1/r4,r1/s5,r6/s6,s1/w3,v1/w5,v2/w4,v4/u3,t1/u5,t6/u6,p6/t2,p5/t3,t4/7',t5/8',p3/5',p4/6',q2/3',q3/4',w1/14',w6/13',v5/11',v6/12',r2/1',r3/2'}
{\path[edge,dashed] (\x) -- (\y);}

\draw[edge, dashed] plot [smooth,tension=0.7] coordinates{(u2) (-1,-4) (-5,-5) (10')};

\draw[edge, dashed] plot [smooth,tension=0.7] coordinates{(u1) (-2,-4) (-6,-5) (9')};

\draw[edge, dashed] plot [smooth,tension=0.7] coordinates{(s2) (7.5,1) (w2)};

\draw[edge, dashed] plot [smooth,tension=0.7] coordinates{(s3) (8.5,2.5) (8.5,0.5) (15')};

\draw[edge, dashed] plot [smooth,tension=0.7] coordinates{(s4) (5.5,3.5)(8,3.5) (9.5,1.5) (16')};

\draw[edge, dashed] plot [smooth,tension=0.7] coordinates{(q6) (-1,0.5) (r5)};

\draw[edge, dashed] plot [smooth,tension=0.7] coordinates{(u4) (-1.5,0) (v3)};

\begin{scope} [scale=1, shift = {(-5, -7.3)}]
\foreach \x/\y/\z in {1/-0.5/0,5/-0.5/1,9/-0.5/2}
{\node[ver] () at (\x,\y){$\z$};}
\path[edge,dashed] (0,0) -- (2,0);
\path[edge] (4,0) -- (6,0);
\draw [line width=3pt, line cap=round, dash pattern=on 0pt off 1.3\pgflinewidth]  (4,0) -- (6,0);
\path[edge] (8,0) -- (10,0);

\end{scope}

\end{tikzpicture}

 \caption{Embedding on $\#_3 \mathbb{RP}^2$ of gem representing $\#_{3} \mathbb{RP}^2$ of type $(4,6,16)$.}  \label{fig:7}
\end{figure}
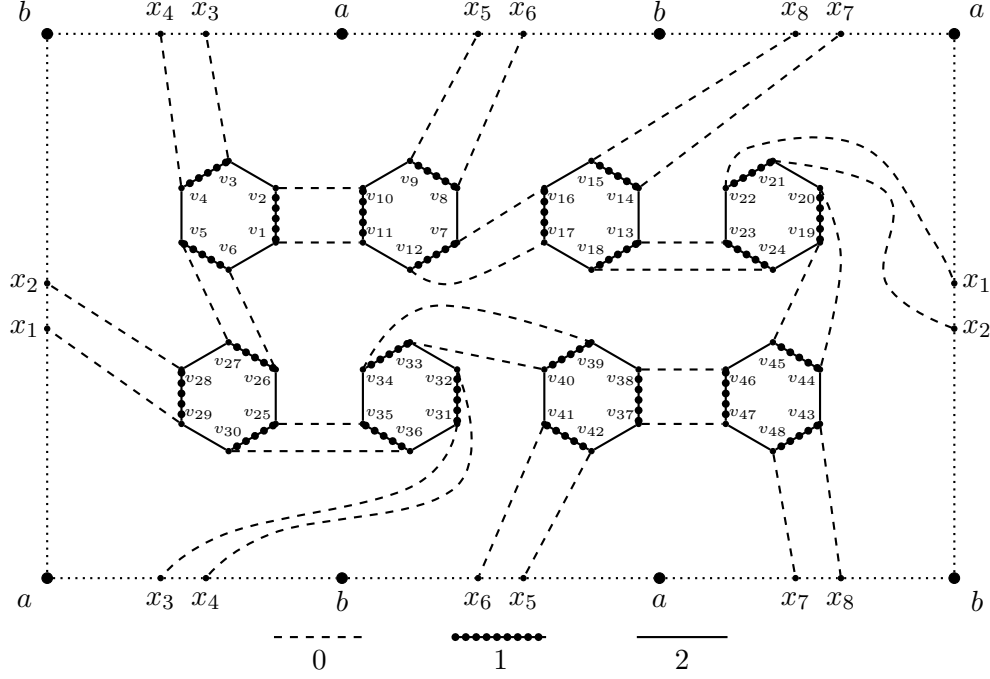

\noindent \textbf{Figure \ref{fig:8}.}  In this illustration, the surface contains two $0$-cells labeled $a$ and $b$, four $1$-cells, and one $2$-cell. Therefore, the Euler characteristic of the surface is $-1$. The $\{1,2\}$-colored six $6$-cycles $A_1,A_2,\dots, A_6$ bound the hexagons $v_1v_2v_3v_4v_5v_6, v_7v_8v_9v_{10}v_{11}v_{12},\dots, v_{31}v_{32}v_{33}v_{34}v_{35}v_{36}$, respectively. The $\{0,2\}$-colored two $18$-cycles $B_1,B_2$ bound the faces $v_3x_3v_{22}v_{23}v_{29}v_{30}v_{34}$ $v_{35}x_2v_{19}v_{24}v_{28}v_{27}v_{10}v_{9}v_{1}v_{6}x_5v_{7}v_8v_2$ containing $a$, $v_4x_4v_{21}v_{20}x_1v_{36}v_{31}v_{18}v_{13}x_7v_{15}v_{14}x_8v_{16}v_{17}v_{32}$ $v_{33}v_{25}v_{26}v_{11}v_{12}x_6v_5$ containing $b$, respectively. The $\{0,1\}$-colored nine $4$-cycles $C_1,C_2,\dots, C_9$ bound the faces $v_2v_{8}v_{9}v_{1}, v_3x_{3}v_{22}v_{21}x_{4}v_{4},\dots, v_5x_{6}v_{12}v_{7}x_{5}v_{6}$, respectively. Thus, the 3-regular colored graph in Figure \ref{fig:8} is a semi-equivelar gem of type $[(4,6,18);36]$ embedded regularly on the surface $\#_3 \mathbb{RP}^2$.

\begin{figure}[h!]
\tikzstyle{ver}=[]
\tikzstyle{vert}=[circle, draw, fill=black!100, inner sep=0pt, minimum width=2pt]
\tikzstyle{verti}=[circle, draw, fill=black!100, inner sep=0pt, minimum width=4pt]
\tikzstyle{edge} = [draw,thick,-]
    \centering
\begin{tikzpicture}[scale=0.6]

\begin{scope}
\foreach \x/\y/\z in
{10/6/1,3.5/6/2,-3.5/6/3,-10/6/4,-10/0/5,-10/-6/6,0/-6/7,10/-6/8}
{\node[verti] (\z) at (\x,\y){};}

\foreach \x/\y/\z in
{10.5/6.5/b,3.5/6.5/b,-3.5/6.5/a,-10.5/6.5/b,-10.5/0/a,-10.5/-6.5/b,0/-6.5/a,10.5/-6.5/b}
{\node[ver] (\z) at (\x,\y){$\z$};}

\foreach \x/\y/\z in
{7.5/6/1',6.5/6/2',0.5/6/3',-0.5/6/4',-6.5/6/5',-7.5/6/6',-10/3.5/7',-10/2.5/8',-10/-2.5/9',-10/-3.5/10',-5.5/-6/11',-4.5/-6/12',4.5/-6/13',5.5/-6/14',10/2.5/15',10/3.5/16'}
{\node[vert] (\z) at (\x,\y){};}

\foreach \x/\y/\z in
{7.5/6.5/x_7,6.5/6.5/x_8,0.5/6.5/x_6,-0.5/6.5/x_5,-6.5/6.5/x_3,-7.5/6.5/x_4,-10.5/3.5/x_6,-10.5/2.5/x_5,-10.5/-2.5/x_2,-10.5/-3.5/x_1,-5.5/-6.5/x_4,-4.5/-6.5/x_3,4.5/-6.5/x_2,5.5/-6.5/x_1,10.5/2.5/x_7,10.5/3.5/x_8}
{\node[ver] (\z) at (\x,\y){$\z$};}

\end{scope}

\begin{scope}[shift={(-5.5,2)},rotate=-30]
\foreach \x/\y in {60/p1,120/p2,180/p3,240/p4,300/p5,360/p6}
{\node[vert] (\y) at (\x:1.3){};}

\foreach \x/\y in {60/v_1,120/v_2,180/v_3,240/v_4,300/v_5,360/v_6}
{\node[ver] at (\x:0.8){\footnotesize{$\y$}};}
\end{scope}

\begin{scope}[shift={(0,2)},rotate=-30]
\foreach \x/\y in {60/q1,120/q2,180/q3,240/q4,300/q5,360/q6}
{\node[vert] (\y) at (\x:1.3){};}

\foreach \x/\y in {60/v_7,120/v_8,180/v_9,240/v_{10},300/v_{11},360/v_{12}}
{\node[ver] at (\x:0.8){\footnotesize{$\y$}};}
\end{scope}

\begin{scope}[shift={(5.5,2)},rotate=-30]
\foreach \x/\y in {60/r1,120/r2,180/r3,240/r4,300/r5,360/r6}
{\node[vert] (\y) at (\x:1.3){};}

\foreach \x/\y in {60/v_{13},120/v_{14},180/v_{15},240/v_{16},300/v_{17},360/v_{18}}
{\node[ver] at (\x:0.8){\footnotesize{$\y$}};}
\end{scope}

\begin{scope}[shift={(-5.5,-2)},rotate=-30]
\foreach \x/\y in {60/s1,120/s2,180/s3,240/s4,300/s5,360/s6}
{\node[vert] (\y) at (\x:1.3){};}

\foreach \x/\y in {60/v_{19},120/v_{20},180/v_{21},240/v_{22},300/v_{23},360/v_{24}}
{\node[ver] at (\x:0.8){\footnotesize{$\y$}};}
\end{scope}

\begin{scope}[shift={(0,-2)},rotate=-30]
\foreach \x/\y in {60/t1,120/t2,180/t3,240/t4,300/t5,360/t6}
{\node[vert] (\y) at (\x:1.3){};}

\foreach \x/\y in {60/v_{25},120/v_{26},180/v_{27},240/v_{28},300/v_{29},360/v_{30}}
{\node[ver] at (\x:0.8){\footnotesize{$\y$}};}
\end{scope}

\begin{scope}[shift={(5.5,-2)},rotate=-30]
\foreach \x/\y in {60/u1,120/u2,180/u3,240/u4,300/u5,360/u6}
{\node[vert] (\y) at (\x:1.3){};}

\foreach \x/\y in {60/v_{31},120/v_{32},180/v_{33},240/v_{34},300/v_{35},360/v_{36}}
{\node[ver] at (\x:0.8){\scriptsize{$\y$}};}
\end{scope}

\foreach \x/\y in {1/2,2/3,3/4,4/5,5/6,6/7,7/8,8/1}
{\path[edge,dotted] (\x) -- (\y);}

\foreach \x/\y in {p1/p2,p2/p3,p3/p4,p4/p5,p5/p6,p6/p1}
{\path[edge] (\x) -- (\y);}

\foreach \x/\y in {p1/p2,p3/p4,p5/p6}
{{\draw [line width=3pt, line cap=round, dash pattern=on 0pt off 1.3\pgflinewidth]  (\x) -- (\y);}}

\foreach \x/\y in {q1/q2,q2/q3,q3/q4,q4/q5,q5/q6,q6/q1}
{\path[edge] (\x) -- (\y);}

\foreach \x/\y in {q1/q6,q2/q3,q4/q5}
{{\draw [line width=3pt, line cap=round, dash pattern=on 0pt off 1.3\pgflinewidth]  (\x) -- (\y);}}

\foreach \x/\y in {r1/r2,r2/r3,r3/r4,r4/r5,r5/r6,r6/r1}
{\path[edge] (\x) -- (\y);}

\foreach \x/\y in {r1/r2,r3/r4,r5/r6}
{{\draw [line width=3pt, line cap=round, dash pattern=on 0pt off 1.3\pgflinewidth]  (\x) -- (\y);}}

\foreach \x/\y in {s1/s2,s2/s3,s3/s4,s4/s5,s5/s6,s6/s1}
{\path[edge] (\x) -- (\y);}

\foreach \x/\y in {s1/s2,s3/s4,s5/s6}
{{\draw [line width=3pt, line cap=round, dash pattern=on 0pt off 1.3\pgflinewidth]  (\x) -- (\y);}}

\foreach \x/\y in {t1/t2,t2/t3,t3/t4,t4/t5,t5/t6,t6/t1}
{\path[edge] (\x) -- (\y);}

\foreach \x/\y in {t1/t6,t2/t3,t4/t5}
{{\draw [line width=3pt, line cap=round, dash pattern=on 0pt off 1.3\pgflinewidth]  (\x) -- (\y);}}

\foreach \x/\y in {u1/u2,u2/u3,u3/u4,u4/u5,u5/u6,u6/u1}
{\path[edge] (\x) -- (\y);}

\foreach \x/\y in {u1/u2,u3/u4,u5/u6}
{{\draw [line width=3pt, line cap=round, dash pattern=on 0pt off 1.3\pgflinewidth]  (\x) -- (\y);}}

\foreach \x/\y in {p1/q3,p2/q2,q4/t3,q5/t2,r6/u1,r5/u2,t1/u3,t6/u4,s6/t4,s5/t5,u5/13',u6/14',r1/15',r2/16',r3/1',q1/4',p3/5',s4/12'}
{\path[edge,dashed] (\x) -- (\y);}

\draw[edge, dashed] plot [smooth,tension=0.7] coordinates{(p4) (-7.5,2.5) (6')};

\draw[edge, dashed] plot [smooth,tension=0.7] coordinates{(s3) (-7.5,-2.5) (11')};

\draw[edge, dashed] plot [smooth,tension=0.7] coordinates{(r4) (3.5,2.5) (2')};

\draw[edge, dashed] plot [smooth,tension=0.7] coordinates{(q6) (2,2.5) (3')};

\draw[edge, dashed] plot [smooth,tension=0.7] coordinates{(p5) (-8,1.5) (7')};

\draw[edge, dashed] plot [smooth,tension=0.7] coordinates{(p6) (-5.5,0.3) (-8,1) (8')};

\draw[edge, dashed] plot [smooth,tension=0.7] coordinates{(s2) (-7.5,-1.3) (10')};

\draw[edge, dashed] plot [smooth,tension=0.7] coordinates{(s1) (-5.5,-0.3) (-8,-1) (9')};

\begin{scope} [scale=1, shift = {(-5, -7.3)}]
\foreach \x/\y/\z in {1/-0.5/0,5/-0.5/1,9/-0.5/2}
{\node[ver] () at (\x,\y){$\z$};}
\path[edge,dashed] (0,0) -- (2,0);
\path[edge] (4,0) -- (6,0);
\draw [line width=3pt, line cap=round, dash pattern=on 0pt off 1.3\pgflinewidth]  (4,0) -- (6,0);
\path[edge] (8,0) -- (10,0);

\end{scope}  

 \end{tikzpicture}

 \caption{Embedding on $\#_3 \mathbb{RP}^2$ of gem representing $\#_{3} \mathbb{RP}^2$ of type $(4,6,18)$.}  \label{fig:8}
\end{figure}
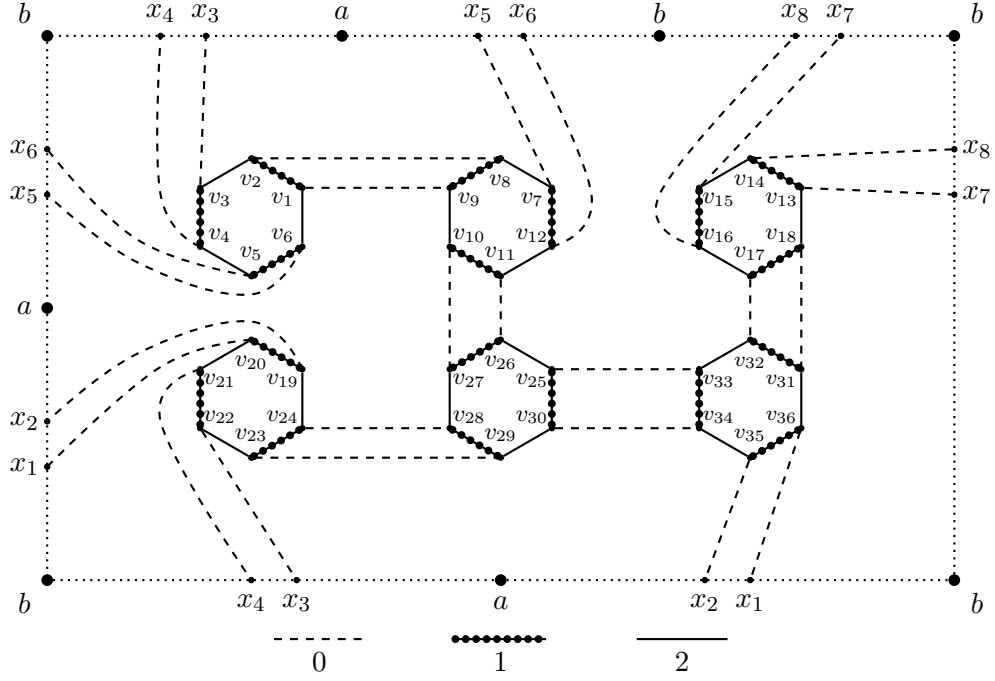

\noindent \textbf{Figure \ref{fig:9}.} In this figure, the surface exhibits only one $0$-cell labeled $a$, three $1$-cells, and one $2$-cell. Hence, the Euler characteristic of the surface is $-1$. The $\{1,2\}$-colored four $6$-cycles $A_1,A_2,A_3,A_4$ bound the hexagonal faces $v_1v_2v_3v_4v_5v_6, v_7v_8v_9v_{10}v_{11}v_{12}, v_{13}v_{14}v_{15}v_{16}v_{17}v_{18}, v_{19}$ $v_{20}v_{21}v_{22}v_{23}v_{24}$, respectively. The $\{0,2\}$-colored $24$-cycle bound the face $v_6x_4v_{8}v_{7}x_5v_{24}v_{19}v_{11}v_{12}$ $x_6v_{23}v_{22}v_{17}v_{16}x_1v_{3}v_{4}v_{14}v_{15}x_2v_{2}v_1x_3v_9v_{10}v_{20}v_{21}v_{18}v_{13}v_5$ containing the $0$-cell $a$. The $\{0,1\}$-colored six $4$-cycles $C_1,C_2,\dots, C_6$ bound the faces $v_5v_{13}v_{14}v_{4}, v_3x_{1}v_{16}v_{15}x_{2}v_{2},\dots, v_1x_{3}v_{9}v_{8}x_{4}v_{6}$, respectively. Thus, the 3-regular colored graph in Figure \ref{fig:9} is a semi-equivelar gem embedded regularly on the surface $\#_3 \mathbb{RP}^2$ of type $[(4,6,24);24]$.

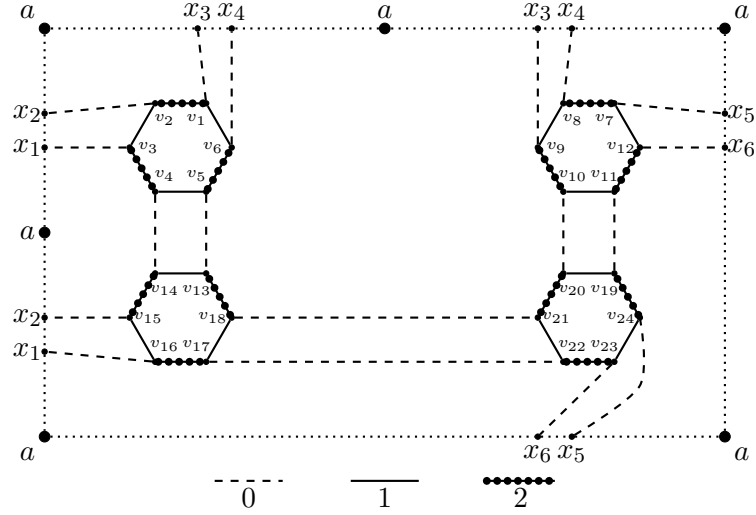
\begin{figure}[h!]
\tikzstyle{ver}=[]
\tikzstyle{vert}=[circle, draw, fill=black!100, inner sep=0pt, minimum width=2pt]
\tikzstyle{verti}=[circle, draw, fill=black!100, inner sep=0pt, minimum width=4pt]
\tikzstyle{edge} = [draw,thick,-]
    \centering
\begin{tikzpicture}[scale=0.45]
\begin{scope}

      \foreach \x/\y/\z in
{10/6/1,0/6/2,-10/6/3,-10/0/4,-10/-6/5,10/-6/6}
{\node[verti] (\z) at (\x,\y){};}

\foreach \x/\y/\z in
{10.5/6.5/a,0/6.5/a,-10.5/6.5/a,-10.5/0/a,-10.5/-6.5/a,10.5/-6.5/a}
{\node[ver] (\z) at (\x,\y){$\z$};}

\foreach \x/\y/\z in
{5.5/6/1',4.5/6/2',-4.5/6/3',-5.5/6/4',-10/3.5/5',-10/2.5/6',-10/-2.5/7',-10/-3.5/8',4.5/-6/9',5.5/-6/10',10/2.5/11',10/3.5/12'}
{\node[vert] (\z) at (\x,\y){};}

\foreach \x/\y/\z in
{5.5/6.5/x_4,4.5/6.5/x_3,-4.5/6.5/x_4,-5.5/6.5/x_3,-10.5/3.5/x_2,-10.5/2.5/x_1,-10.5/-2.5/x_2,-10.5/-3.5/x_1,4.5/-6.5/x_6,5.5/-6.5/x_5,10.5/2.5/x_6,10.5/3.5/x_5}
{\node[ver] (\z) at (\x,\y){$\z$};}

\end{scope}

\begin{scope}[shift={(-6,2.5)}]
\foreach \x/\y in {60/w1,120/w2,180/w3,240/w4,300/w5,360/w6}
{\node[vert] (\y) at (\x:1.5){};}

\foreach \x/\y in {60/v_1,120/v_2,180/v_3,240/v_4,300/v_5,360/v_6}
{\node[ver] at (\x:0.95){\tiny{$\y$}};}
\end{scope}

\begin{scope}[shift={(6,2.5)}]
\foreach \x/\y in {60/x1,120/x2,180/x3,240/x4,300/x5,360/x6}
{\node[vert] (\y) at (\x:1.5){};}

\foreach \x/\y in {60/v_7,120/v_8,180/v_9,240/v_{10},300/v_{11},360/v_{12}}
{\node[ver] at (\x:0.95){\tiny{$\y$}};}
\end{scope}

\begin{scope}[shift={(-6,-2.5)}]
\foreach \x/\y in {60/y1,120/y2,180/y3,240/y4,300/y5,360/y6}
{\node[vert] (\y) at (\x:1.5){};}

\foreach \x/\y in {60/v_{13},120/v_{14},180/v_{15},240/v_{16},300/v_{17},360/v_{18}}
{\node[ver] at (\x:0.95){\tiny{$\y$}};}
\end{scope}

\begin{scope}[shift={(6,-2.5)}]
\foreach \x/\y in {60/z1,120/z2,180/z3,240/z4,300/z5,360/z6}
{\node[vert] (\y) at (\x:1.5){};}

\foreach \x/\y in {60/v_{19},120/v_{20},180/v_{21},240/v_{22},300/v_{23},360/v_{24}}
{\node[ver] at (\x:0.95){\tiny{$\y$}};}
\end{scope}

\foreach \x/\y in {1/1',1'/2',2'/2,2/3',3'/4',4'/3,3/5',5'/6',6'/4,4/7',7'/8',8'/5,5/9',9'/10',10'/6,6/11',11'/12',12'/1}
{\path[edge,dotted] (\x) -- (\y);}

\foreach \x/\y in {x1/x2,x2/x3,x3/x4,x4/x5,x5/x6,x6/x1}
{\path[edge] (\x) -- (\y);}

\foreach \x/\y in {x1/x2,x3/x4,x5/x6}
{{\draw [line width=3pt, line cap=round, dash pattern=on 0pt off 1.3\pgflinewidth]  (\x) -- (\y);}}

\foreach \x/\y in {w1/w2,w2/w3,w3/w4,w4/w5,w5/w6,w6/w1}
{\path[edge] (\x) -- (\y);}

\foreach \x/\y in {w1/w2,w3/w4,w5/w6}
{{\draw [line width=3pt, line cap=round, dash pattern=on 0pt off 1.3\pgflinewidth]  (\x) -- (\y);}}

\foreach \x/\y in {y1/y2,y2/y3,y3/y4,y4/y5,y5/y6,y6/y1}
{\path[edge] (\x) -- (\y);}

\foreach \x/\y in {y2/y3,y4/y5,y6/y1}
{{\draw [line width=3pt, line cap=round, dash pattern=on 0pt off 1.3\pgflinewidth]  (\x) -- (\y);}}

\foreach \x/\y in {z1/z2,z2/z3,z3/z4,z4/z5,z5/z6,z6/z1}
{\path[edge] (\x) -- (\y);}

\foreach \x/\y in {z2/z3,z4/z5,z6/z1}
{{\draw [line width=3pt, line cap=round, dash pattern=on 0pt off 1.3\pgflinewidth]  (\x) -- (\y);}}

\foreach \x/\y in {w4/y2,w5/y1,y6/z3,y5/z4,z1/x5,z2/x4,x6/11',x1/12',x2/1',x3/2',w6/3',w1/4',w2/5',w3/6',y3/7',y4/8',z5/9'}
{\path[edge,dashed] (\x) -- (\y);}


\draw[edge, dashed] plot [smooth,tension=0.5] coordinates{(z6) (7.5,-4.5) (10')};

\begin{scope} [scale=1, shift = {(-5, -7.3)}]
\foreach \x/\y/\z in {1/-0.5/0,5/-0.5/1,9/-0.5/2}
{\node[ver] () at (\x,\y){$\z$};}
\path[edge,dashed] (0,0) -- (2,0);
\path[edge] (4,0) -- (6,0);

\path[edge] (8,0) -- (10,0);
\draw [line width=3pt, line cap=round, dash pattern=on 0pt off 1.3\pgflinewidth]  (8,0) -- (10,0);

\end{scope}  

\end{tikzpicture}

 \caption{Embedding on $\#_3 \mathbb{RP}^2$ of gem representing $\#_{3} \mathbb{RP}^2$ of type $(4,6,24)$.}  \label{fig:9}
\end{figure}

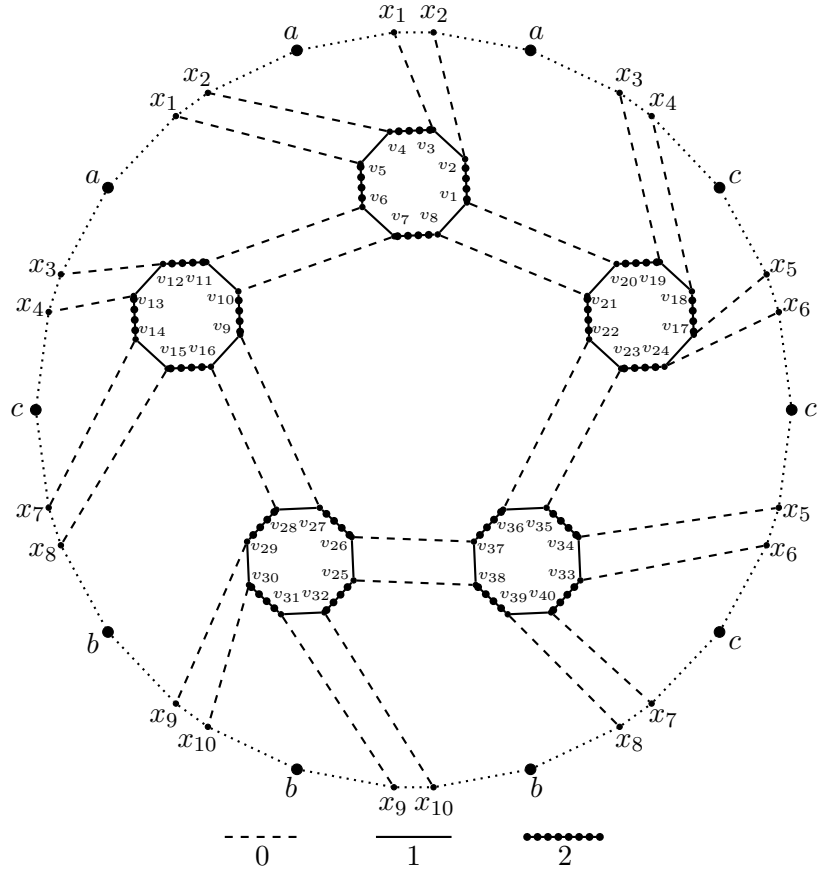
\begin{figure}[h!]
\tikzstyle{ver}=[]
\tikzstyle{vert}=[circle, draw, fill=black!100, inner sep=0pt, minimum width=2pt]
\tikzstyle{verti}=[circle, draw, fill=black!100, inner sep=0pt, minimum width=4pt]
\tikzstyle{edge} = [draw,thick,-]
    \centering
\begin{tikzpicture}[scale=0.5]

\begin{scope}
    \foreach \x/\y in {0/1,36/2,72/3,108/4,144/5,180/6,216/7,252/8,288/9,324/10}
    {\node[verti] (\y) at (\x:10){};}

    \foreach \x/\y in {0/c,36/c,72/a,108/a,144/a,180/c,216/b,252/b,288/b,324/c}
    {\node[ver] (\y) at (\x:10.5){$\y$};}

 \foreach \x/\y in {15/1',21/2',51/3',57/4',87/5',93/6',123/7',129/8',159/9',165/10',195/11',201/12',231/13',237/14',267/15',273/16',303/17',309/18',339/19',345/20'}
    {\node[vert] (\y) at (\x:10){};}

    \foreach \x/\y in {15/x_6,21/x_5,51/x_4,57/x_3,87/x_2,93/x_1,123/x_2,129/x_1,159/x_3,165/x_4,195/x_7,201/x_8,231/x_9,237/x_{10},267/x_9,273/x_{10},303/x_8,309/x_7,339/x_6,345/x_5}
    {\node[ver] (\y) at (\x:10.5){$\y$};}
    
\end{scope}

\begin{scope}[shift={(-6,2.5)},rotate=-20]

\foreach \x/\y in {0/x1,45/x2,90/x3,135/x4,180/x5,225/x6,270/x7,315/x8}
{\node[vert] (\y) at (\x:1.5){};}

\foreach \x/\y in {0/v_9,45/v_{10},90/v_{11},135/v_{12},180/v_{13},225/v_{14},270/v_{15},315/v_{16}}
{\node[ver] at (\x:1){\tiny{$\y$}};}

\end{scope}

\begin{scope}[shift={(6,2.5)},rotate=-20]
\foreach \x/\y in {0/y1,45/y2,90/y3,135/y4,180/y5,225/y6,270/y7,315/y8}
{\node[vert] (\y) at (\x:1.5){};}

\foreach \x/\y in {0/v_{17},45/v_{18},90/v_{19},135/v_{20},180/v_{21},225/v_{22},270/v_{23},315/v_{24}}
{\node[ver] at (\x:1){\tiny{$\y$}};}
\end{scope}

\begin{scope}[shift={(-3,-4)},rotate=-20]
\foreach \x/\y in {0/z1,45/z2,90/z3,135/z4,180/z5,225/z6,270/z7,315/z8}
{\node[vert] (\y) at (\x:1.5){};}

\foreach \x/\y in {0/v_{25},45/v_{26},90/v_{27},135/v_{28},180/v_{29},225/v_{30},270/v_{31},315/v_{32}}
{\node[ver] at (\x:1){\tiny{$\y$}};}
\end{scope}

\begin{scope}[shift={(3,-4)},rotate=-20]
\foreach \x/\y in {0/d1,45/d2,90/d3,135/d4,180/d5,225/d6,270/d7,315/d8}
{\node[vert] (\y) at (\x:1.5){};}

\foreach \x/\y in {0/v_{33},45/v_{34},90/v_{35},135/v_{36},180/v_{37},225/v_{38},270/v_{39},315/v_{40}}
{\node[ver] at (\x:1){\tiny{$\y$}};}
\end{scope}

\begin{scope}[shift={(0,6)},rotate=-20]
\foreach \x/\y in {0/u1,45/u2,90/u3,135/u4,180/u5,225/u6,270/u7,315/u8}
{\node[vert] (\y) at (\x:1.5){};}

\foreach \x/\y in {0/v_{1},45/v_{2},90/v_{3},135/v_{4},180/v_{5},225/v_{6},270/v_{7},315/v_{8}}
{\node[ver] at (\x:1){\tiny{$\y$}};}
\end{scope}

\foreach \x/\y in {1/1',1'/2',2'/2,2/3',3'/4',4'/3,3/5',5'/6',6'/4,4/7',7'/8',8'/5,5/9',9'/10',10'/6,6/11',11'/12',12'/7,7/13',13'/14',14'/8,8/15',15'/16',16'/9,9/17',17'/18',18'/10,10/19',19'/20',20'/1}
{\path[edge,dotted] (\x) -- (\y);}

\foreach \x/\y in {x1/x2,x2/x3,x3/x4,x4/x5,x5/x6,x6/x7,x7/x8,x8/x1}
{\path[edge] (\x) -- (\y);}

\foreach \x/\y in {x1/x2,x3/x4,x5/x6,x7/x8}
{{\draw [line width=3pt, line cap=round, dash pattern=on 0pt off 1.3\pgflinewidth]  (\x) -- (\y);}}

\foreach \x/\y in {y1/y2,y2/y3,y3/y4,y4/y5,y5/y6,y6/y7,y7/y8,y8/y1}
{\path[edge] (\x) -- (\y);}

\foreach \x/\y in {y1/y2,y3/y4,y5/y6,y7/y8}
{{\draw [line width=3pt, line cap=round, dash pattern=on 0pt off 1.3\pgflinewidth]  (\x) -- (\y);}}

\foreach \x/\y in {z1/z2,z2/z3,z3/z4,z4/z5,z5/z6,z6/z7,z7/z8,z8/z1}
{\path[edge] (\x) -- (\y);}

\foreach \x/\y in {z2/z3,z4/z5,z6/z7,z8/z1}
{{\draw [line width=3pt, line cap=round, dash pattern=on 0pt off 1.3\pgflinewidth]  (\x) -- (\y);}}

\foreach \x/\y in {d1/d2,d2/d3,d3/d4,d4/d5,d5/d6,d6/d7,d7/d8,d8/d1}
{\path[edge] (\x) -- (\y);}

\foreach \x/\y in {d2/d3,d4/d5,d6/d7,d8/d1}
{{\draw [line width=3pt, line cap=round, dash pattern=on 0pt off 1.3\pgflinewidth]  (\x) -- (\y);}}

\foreach \x/\y in {u1/u2,u2/u3,u3/u4,u4/u5,u5/u6,u6/u7,u7/u8,u8/u1}
{\path[edge] (\x) -- (\y);}

\foreach \x/\y in {u1/u2,u3/u4,u5/u6,u7/u8}
{{\draw [line width=3pt, line cap=round, dash pattern=on 0pt off 1.3\pgflinewidth]  (\x) -- (\y);}}

\foreach \x/\y in {u1/y4,u8/y5,u7/x2,u6/x3,x1/z3,x8/z4,z1/d6,z2/d5,y7/d3,y6/d4,u2/5',u3/6',u4/7',u5/8',x4/9',x5/10',x6/11',x7/12',z5/13',z6/14',z7/15',z8/16',d7/17',d8/18',d1/19',d2/20',y8/1',y1/2',y2/3',y3/4'}
{\path[edge,dashed] (\x) -- (\y);}

\begin{scope} [scale=1, shift = {(-5, -11.3)}]
\foreach \x/\y/\z in {1/-0.5/0,5/-0.5/1,9/-0.5/2}
{\node[ver] () at (\x,\y){$\z$};}
\path[edge,dashed] (0,0) -- (2,0);
\path[edge] (4,0) -- (6,0);

\path[edge] (8,0) -- (10,0);
\draw [line width=3pt, line cap=round, dash pattern=on 0pt off 1.3\pgflinewidth]  (8,0) -- (10,0);

\end{scope}

\end{tikzpicture}

 \caption{Embedding on $\#_3 \mathbb{RP}^2$ of gem representing $\#_{3} \mathbb{RP}^2$ of type $(4,8,10)$.}  \label{fig:10}
\end{figure}

\begin{figure}[h!]
\tikzstyle{ver}=[]
\tikzstyle{vert}=[circle, draw, fill=black!100, inner sep=0pt, minimum width=2pt]
\tikzstyle{verti}=[circle, draw, fill=black!100, inner sep=0pt, minimum width=4pt]
\tikzstyle{edge} = [draw,thick,-]
    \centering
\begin{tikzpicture}[scale=0.6]

\begin{scope}
\foreach \x/\y/\z in
{10/5/1,3.5/5/2,-3.5/5/3,-10/5/4,-10/-5/5,-3.5/-5/6,3.5/-5/7,10/-5/8}
{\node[verti] (\z) at (\x,\y){};}

\foreach \x/\y/\z in
{10.5/5.5/a,3.5/5.5/b,-3.5/5.5/a,-10.5/5.5/b,-10.5/-5.5/a,-3.5/-5.5/b,3.5/-5.5/a,10.5/-5.5/b}
{\node[ver] (\z) at (\x,\y){$\z$};}

\foreach \x/\y/\z in
{7.5/5/1',6.5/5/2',0.5/5/3',-0.5/5/4',-6.5/5/5',-7.5/5/6',-10/0.5/7',-10/-0.5/8',-7.5/-5/9',-6.5/-5/10',-0.5/-5/11',0.5/-5/12',6.5/-5/13',7.5/-5/14',10/-0.5/15',10/0.5/16'}
{\node[vert] (\z) at (\x,\y){};}

\foreach \x/\y/\z in
{7.5/5.5/x_7,6.5/5.5/x_8,0.5/5.5/x_6,-0.5/5.5/x_5,-6.5/5.5/x_3,-7.5/5.5/x_4,-10.5/0.5/x_2,-10.5/-0.5/x_1,-7.5/-5.5/x_3,-6.5/-5.5/x_4,-0.5/-5.5/x_6,0.5/-5.5/x_5,6.5/-5.5/x_7,7.5/-5.5/x_8,10.5/-0.5/x_{2},10.5/0.5/x_{1}}
{\node[ver] (\z) at (\x,\y){$\z$};}

\end{scope}

\begin{scope}[shift={(-6,0)},rotate=-20]
\foreach \x/\y in {0/x1,45/x2,90/x3,135/x4,180/x5,225/x6,270/x7,315/x8}
{\node[vert] (\y) at (\x:1.5){};}

\foreach \x/\y in {0/v_1,45/v_2,90/v_3,135/v_4,180/v_5,225/v_6,270/v_7,315/v_8}
{\node[ver] at (\x:1){\tiny{$\y$}};}
\end{scope}

\begin{scope}[shift={(0,0)},rotate=-20]
\foreach \x/\y in {0/y1,45/y2,90/y3,135/y4,180/y5,225/y6,270/y7,315/y8}
{\node[vert] (\y) at (\x:1.5){};}

\foreach \x/\y in {0/v_9,45/v_{10},90/v_{11},135/v_{12},180/v_{13},225/v_{14},270/v_{15},315/v_{16}}
{\node[ver] at (\x:1){\tiny{$\y$}};}
\end{scope}

\begin{scope}[shift={(6,0)},rotate=-20]
\foreach \x/\y in {0/z1,45/z2,90/z3,135/z4,180/z5,225/z6,270/z7,315/z8}
{\node[vert] (\y) at (\x:1.5){};}

\foreach \x/\y in {0/v_{17},45/v_{18},90/v_{19},135/v_{20},180/v_{21},225/v_{22},270/v_{23},315/v_{24}}
{\node[ver] at (\x:1){\tiny{$\y$}};}
\end{scope}

\foreach \x/\y in {1/1',1'/2',2'/2,2/3',3'/4',4'/3,3/5',5'/6',6'/4,4/7',7'/8',8'/5,5/9',9'/10',10'/6,6/11',11'/12',12'/7,7/13',13'/14',14'/8,8/15',15'/16',16'/1}
{\path[edge,dotted] (\x) -- (\y);}

\foreach \x/\y in {x1/x2,x2/x3,x3/x4,x4/x5,x5/x6,x6/x7,x7/x8,x8/x1}
{\path[edge] (\x) -- (\y);}

\foreach \x/\y in {x1/x2,x3/x4,x5/x6,x7/x8}
{{\draw [line width=3pt, line cap=round, dash pattern=on 0pt off 1.3\pgflinewidth]  (\x) -- (\y);}}

\foreach \x/\y in {y1/y2,y2/y3,y3/y4,y4/y5,y5/y6,y6/y7,y7/y8,y8/y1}
{\path[edge] (\x) -- (\y);}

\foreach \x/\y in {y1/y2,y3/y4,y5/y6,y7/y8}
{{\draw [line width=3pt, line cap=round, dash pattern=on 0pt off 1.3\pgflinewidth]  (\x) -- (\y);}}

\foreach \x/\y in {z1/z2,z2/z3,z3/z4,z4/z5,z5/z6,z6/z7,z7/z8,z8/z1}
{\path[edge] (\x) -- (\y);}

\foreach \x/\y in {z1/z2,z3/z4,z5/z6,z7/z8}
{{\draw [line width=3pt, line cap=round, dash pattern=on 0pt off 1.3\pgflinewidth]  (\x) -- (\y);}}

\foreach \x/\y in {x2/y5,x3/y4,y1/z6,y8/z7,x4/5',x5/6',x6/7',x7/8',x8/9',x1/10',y2/3',y3/4',y6/11',y7/12',z8/13',z1/14',z2/15',z3/16',z4/1',z5/2'}
{\path[edge,dashed] (\x) -- (\y);}

\begin{scope} [scale=1, shift = {(-5, -7)}]
\foreach \x/\y/\z in {1/-0.5/0,5/-0.5/1,9/-0.5/2}
{\node[ver] () at (\x,\y){$\z$};}
\path[edge,dashed] (0,0) -- (2,0);
\path[edge] (4,0) -- (6,0);

\path[edge] (8,0) -- (10,0);
\draw [line width=3pt, line cap=round, dash pattern=on 0pt off 1.3\pgflinewidth]  (8,0) -- (10,0);

\end{scope} 

\end{tikzpicture}

 \caption{Embedding on $\#_3 \mathbb{RP}^2$ of gem representing $\#_{3} \mathbb{RP}^2$ of type $(4,8,12)$.}  \label{fig:11}
\end{figure}
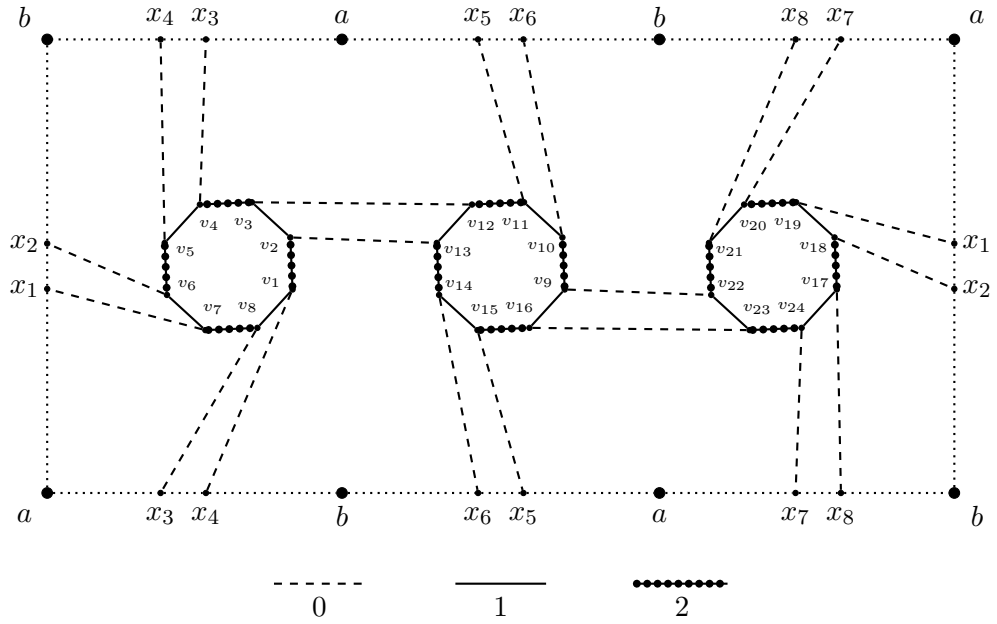

\begin{figure}[h!]
\tikzstyle{ver}=[]
\tikzstyle{vert}=[circle, draw, fill=black!100, inner sep=0pt, minimum width=2pt]
\tikzstyle{verti}=[circle, draw, fill=black!100, inner sep=0pt, minimum width=4pt]
\tikzstyle{edge} = [draw,thick,-]
    \centering
\begin{tikzpicture}[scale=0.45]
\begin{scope}

      \foreach \x/\y/\z in
{10/6/1,0/6/2,-10/6/3,-10/-6/4,0/-6/5,10/-6/6}
{\node[verti] (\z) at (\x,\y){};}

\foreach \x/\y/\z in
{10.5/6.5/a,0/6.5/a,-10.5/6.5/a,-10.5/-6.5/a,0/-6.5/a,10.5/-6.5/a}
{\node[ver] (\z) at (\x,\y){$\z$};}

\foreach \x/\y/\z in
{5.5/6/1',4.5/6/2',-4.5/6/3',-5.5/6/4',-10/0.5/5',-10/-0.5/6',-5.5/-6/7',-4.5/-6/8',4.5/-6/9',5.5/-6/10',10/-0.5/11',10/0.5/12'}
{\node[vert] (\z) at (\x,\y){};}

\foreach \x/\y/\z in
{5.5/6.5/x_1,4.5/6.5/x_2,-4.5/6.5/x_3,-5.5/6.5/x_4,-10.5/0.5/x_3,-10.5/-0.5/x_4,-5.5/-6.5/x_5,-4.5/-6.5/x_6,4.5/-6.5/x_5,5.5/-6.5/x_6,10.5/-0.5/x_1,10.5/0.5/x_2}
{\node[ver] (\z) at (\x,\y){$\z$};}

    \end{scope}

\begin{scope}[shift={(-3,0)},rotate=-25]
 \foreach \x/\y in {0/x1,45/x2,90/x3,135/x4,180/x5,225/x6,270/x7,315/x8}
 {\node[vert] (\y) at (\x:2){};}

  \foreach \x/\y in {0/v_1,45/v_2,90/v_3,135/v_4,180/v_5,225/v_6,270/v_7,315/v_8}
 {\node[ver] at (\x:1.5){\tiny{$\y$}};}
\end{scope}

\begin{scope}[shift={(3,0)},rotate=-20]
 \foreach \x/\y in {0/y1,45/y2,90/y3,135/y4,180/y5,225/y6,270/y7,315/y8}
 {\node[vert] (\y) at (\x:2){};}

  \foreach \x/\y in {0/v_9,45/v_{10},90/v_{11},135/v_{12},180/v_{13},225/v_{14},270/v_{15},315/v_{16}}
 {\node[ver] at (\x:1.5){\tiny{$\y$}};}
\end{scope}

\foreach \x/\y in {1/1',1'/2',2'/2,2/3',3'/4',4'/3,3/5',5'/6',6'/4,4/7',7'/8',8'/5,5/9',9'/10',10'/6,6/11',11'/12',12'/1}
{\path[edge,dotted] (\x) -- (\y);}

\foreach \x/\y in {x1/x2,x2/x3,x3/x4,x4/x5,x5/x6,x6/x7,x7/x8,x8/x1}
{\path[edge] (\x) -- (\y);}

\foreach \x/\y in {x1/x2,x3/x4,x5/x6,x7/x8}
{{\draw [line width=3pt, line cap=round, dash pattern=on 0pt off 1.3\pgflinewidth]  (\x) -- (\y);}}

\foreach \x/\y in {y1/y2,y2/y3,y3/y4,y4/y5,y5/y6,y6/y7,y7/y8,y8/y1}
{\path[edge] (\x) -- (\y);}

\foreach \x/\y in {y1/y2,y3/y4,y5/y6,y7/y8}
{{\draw [line width=3pt, line cap=round, dash pattern=on 0pt off 1.3\pgflinewidth]  (\x) -- (\y);}}

\foreach \x/\y in {x1/y6,x2/y5,x3/3',x4/4',x5/5',x6/6',x7/7',x8/8',y7/9',y8/10',y1/11',y2/12',y3/1',y4/2'}
{\path[edge,dashed] (\x) -- (\y);}

\begin{scope} [scale=1, shift = {(-5, -7.5)}]
\foreach \x/\y/\z in {1/-0.5/0,5/-0.5/1,9/-0.5/2}
{\node[ver] () at (\x,\y){$\z$};}
\path[edge,dashed] (0,0) -- (2,0);
\path[edge] (8,0) -- (10,0);

\path[edge] (4,0) -- (6,0);
\draw [line width=3pt, line cap=round, dash pattern=on 0pt off 1.3\pgflinewidth]  (4,0) -- (6,0);

\end{scope} 

\end{tikzpicture}

 \caption{Embedding on $\#_3 \mathbb{RP}^2$ of gem representing $\#_{3} \mathbb{RP}^2$ of type $(4,8,16)$.}  \label{fig:12}
\end{figure}
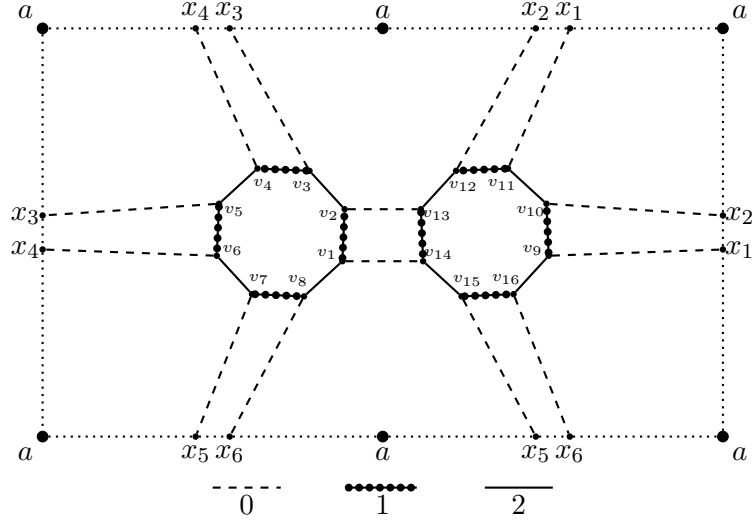

\noindent \textbf{Figure \ref{fig:10}.} In this diagram, the surface is composed of three $0$-cells labeled $a,b$, and $c$, five $1$-cells, and one $2$-cell. Consequently, the Euler characteristic of the surface is $-1$. The $\{1,2\}$-colored five $8$-cycles $A_1,A_2,\dots, A_5$ bound the octagonal faces $v_1v_2v_3v_4v_5v_6v_7v_8, v_9v_{10}v_{11}v_{12}$ $v_{13}v_{14}v_{15}v_{16}, \dots, v_{33}v_{34}v_{35}v_{36}v_{37}v_{38}v_{39}v_{40}$, respectively. One of the four $\{0,2\}$-colored $10$-cycles bound the face  $v_4x_2v_{2}v_{1}v_{20}v_{19}x_3v_{12}v_{11}v_{6}v_{5}x_1v_3$ containing $a$ in its interior. One can look for the other three in Figure \ref{fig:10}. The $\{0,1\}$-colored ten $4$-cycles $C_1,C_2,\dots, C_{10}$ bound the faces $v_1v_{20}v_{21}v_{8}, v_3x_{1}v_{5}v_{4}x_{2}v_{2},\dots, v_{17}x_{5}v_{34}v_{33}x_{6}v_{24}$, respectively. Thus, the 3-regular colored graph in Figure \ref{fig:10} is a semi-equivelar gem of type $[(4,8,10);40]$ embedded regularly on the surface $\#_3 \mathbb{RP}^2$.

\noindent \textbf{Figure \ref{fig:11}.} In this illustration, the surface contains two $0$-cells labeled $a$ and $b$, four $1$-cells, and one $2$-cell. Therefore, the Euler characteristic of the surface is $-1$. The $\{1,2\}$-colored three $8$-cycles $A_1, A_2, A_3$ bound the octagonal faces $v_1v_2v_3v_4v_5v_6v_7v_8, v_9v_{10}v_{11}v_{12}v_{13}v_{14}v_{15}v_{16}, v_{17}v_{18}$ $v_{19}v_{20}v_{21}v_{22}v_{23}v_{24}$, respectively. The two $\{0,2\}$-colored $12$-cycles bound the faces  $v_4x_3v_{8}v_{7}x_1$ $v_{19}v_{20}x_7v_{24}v_{23}v_{16}v_{15}x_5v_{11}v_{12}v_3$ containing $a$ in its interior, $v_1x_4v_{5}v_{6}x_2v_{18}v_{17}x_8v_{21}v_{22}v_{9}v_{10}x_6$ $v_{14}v_{13}v_2$ containing $b$ in its interior, respectively. The $\{0,1\}$-colored six $4$-cycles $C_1,C_2,\dots, C_{6}$ bound the faces $v_2v_{13}v_{12}v_{3}, v_1x_{4}v_{5}v_{4}x_{3}v_{8},\dots, v_{17}x_{8}v_{21}v_{20}x_{7}v_{24}$, respectively. Thus, the 3-regular colored graph in Figure \ref{fig:11} is a semi-equivelar gem of type $[(4,8,12);24]$ embedded regularly on the surface $\#_3 \mathbb{RP}^2$.

\noindent \textbf{Figure \ref{fig:12}.} In this diagram, the surface features only one $0$-cell labeled $a$, three $1$-cells, and one $2$-cell. Consequently, the Euler characteristic of the surface is $-1$. The $\{1,2\}$-colored two $8$-cycles $A_1, A_2$ bound the octagonal faces $v_1v_2v_3v_4v_5v_6v_7v_8, v_9v_{10}v_{11}v_{12}v_{13}v_{14}v_{15}v_{16}$, respectively. The $\{0,2\}$-colored $16$-cycle bound the face  $v_4x_4v_{6}v_{7}x_5v_{15}v_{14}v_1v_8x_6v_{16}v_{9}x_1v_{11}v_{10}x_2v_{12}v_{13}v_2v_3$ $x_3v_5$ containing the $0$-cell $a$ in its interior. The $\{0,1\}$-colored four $4$-cycles $C_1,C_2, C_3, C_{4}$ bound the faces $v_2v_{13}v_{14}v_{1}, v_3x_{3}v_{5}v_{6}x_{4}v_{4}, v_{12}x_{2}v_{10}v_{9}x_{1}v_{11}, v_{7}x_{5}v_{15}v_{16}x_{6}v_{8}$, respectively. Thus, the 3-regular colored graph in Figure \ref{fig:12} is a semi-equivelar gem of type $[(4,8,16);16]$ embedded regularly on the surface $\#_3 \mathbb{RP}^2$.

Thus, Figures \ref{fig:1} through \ref{fig:12} establish the existence of these semi-equivelar gems for each possible type. It follows from Theorem \ref{theorem:possible gem} that these gems represent $\#_3 \mathbb{RP}^2$ itself.
\end{proof}

\begin{remark}
{\rm We emphasize that in defining semi-equivelar graphs, we restricted our focus to cases where the faces, after embedding, are $n$-gons for $n \geq 4$. However, if we extend this to include semi-equivelar gems with 2-gons, Theorem \ref{theorem:possible gem}  still remains valid. This is due to the fact that, in Case 3 of the proof of Lemma \ref{lemma:possible}, Equation \eqref{4} fails to hold when $q_i = 2$ for $0 \leq i \leq l$. Therefore, the possibility of having 2-gons in the regular embedding of a semi-equivelar gem on the surface with Euler characteristic $-1$ can be easily discarded.
    }
\end{remark}

\begin{remark}
{\rm
Let $S$ be a surface with Euler characteristic $-2$. Following a similar argument as given in Lemma \ref{lemma:possible}, we have computed all the possible types of semi-equivelar graphs that can be embedded regularly on the surface $S$. These 31 possibilities are: $(4^5)$, $(6^4)$, $(4^3,6)$, $(4^3,8)$, $(4^3,12)$, $(4,6,4,6)$, $(4^2,6^2)$, $(4,8,4,8)$, $(4^2,8^2)$, $ (8^3)$, $ (10^3)$, $(6^2,8)$, $(6^2,10)$, $(6^2,12)$, $(6^2,18)$, $(10^2,4)$, $(12^2,4)$, $(16^2,4)$, $ (8^2,6)$, $ (12^2,6)$, $(4,6,14)$, $(4,6,16)$, $(4,6,18)$, $(4,6,20)$, $(4,6,24)$, $(4,6,36)$, $(4,8,10)$, $(4,8,12)$, $(4,8,16)$, $(4,8,24)$, and $(4,10,20)$.
It remains for the reader, using a similar construction as in Theorem \ref{thm:construction}, to determine for each of the above types whether there exists a semi-equivelar gem that is embedded regularly on the surface $S$.
    }
\end{remark}

\noindent {\bf Acknowledgement:} The authors wish to express their gratitude to the anonymous referees for their valuable comments and suggestions. Their input has significantly enhanced the quality of this article. The second author is supported by the Mathematical Research Impact Centric Support (MATRICS) Research Grant (MTR/2022/000036) by SERB (India).
				
\medskip

\noindent {\bf Data availability:} The authors declare that all data supporting the findings of this study are available within the article.

\smallskip

\noindent {\bf Declarations}

\smallskip

\noindent {\bf Conflict of interest:} No potential conflict of interest was reported by the authors.

{\footnotesize

\end{document}